\documentclass [a4paper,12pt]{article}

\usepackage{amsthm}
\usepackage{amsmath}
\usepackage{amssymb}
\usepackage{enumerate}
\usepackage{color}

\newcommand{\Nbb}{\ensuremath{\mathbb{N}}}
\newcommand{\Zbb}{\ensuremath{\mathbb{Z}}}

\newcommand{\Rbb}{\ensuremath{\mathbb{R}}}
\newcommand{\Cbb}{\ensuremath{\mathbb{C}}}
\newcommand{\ddfrac}[2]{\displaystyle\frac{#1}{#2}}

\renewcommand{\Im}{\mathop\mathrm{Im}}

\newcommand{\diam}{\mathop\mathrm{diam}}
\newcommand{\na}{\mathop\mathrm{\nabla}}

\newcommand{\Log}{\mathop\mathrm{Log}}

\newcommand{\Spt}{\mathop\mathrm{Supp}}
\newcommand{\Lip}{\mathop\mathrm{Lip}}
\newcommand{\spt}{\mathop\mathrm{Supp}}

\newcommand{\fiproof}{{\hspace*{\fill} $\square$ \vspace{2pt}}}

\def\e{\varepsilon}
\def\pa{\partial}
\def\ov{\overline}

\def\sup{\mathop{\rm sup}}

\def\max{\mathop{\rm max}}
\def\inf{\mathop{\rm inf}}
\def\lim{\mathop{\rm lim}}

\def\D{\Delta}
\def\f{\varphi}
\def\d{\delta}
\def\al{\alpha}
\def\al{\alpha}
\def\ka{\kappa}
\def\la{\lambda}
\def\La{\Lambda}
\def\l{\lambda}

\def\om{\omega}

\def\Ga{\Gamma}
\def\ga{\gamma}
\newtheorem {theorem} {Theorem}

\newtheorem {corollary}{Corollary}
\newtheorem {lemma} {Lemma}[section]

\newtheorem {theor}[lemma]{Theorem}

\newtheorem {defin} [lemma]{Definition}
\newtheorem {example} [lemma]{Example}

\newtheorem {conj} [theorem]{Conjecture}

\theoremstyle{definition}

\numberwithin{equation}{section}

\title{On $C^1$-approximability of functions by solutions of
second order elliptic equations on plane compact sets and $C$-analytic capacity.}

\author{P.V. Paramonov\thanks{Supported by
the Russian Scientific Foundation (grant No. 17-11-01064).} $\;$
and X. Tolsa \thanks{
Partially supported by 2017-SGR-0395 (Catalonia) and MTM-2016-77635-P (MICINN, Spain).
} }

\begin{document}
\maketitle

{\small Criteria for approximability of functions by solutions of homogeneous second order
elliptic equations (with constant complex coeffi\-ci\-ents) in the norms of the Whitney $C^1$-spaces on compact sets in
$\Rbb^2$ are obtained in terms of the respective $C^1$-capacities.
It is proved that the mentioned $C^1$-capacities are comparable to the classic
$C$-analytic capacity, and so have a proper geometric measure characterization.

\bigskip
Bibliography:  25 titles.

\bigskip
 \noindent
 {\bf Key words: }
Second order homogeneous elliptic operator; $C^1$-approximation;
localization operator of Vitushkin type; $L$-oscillation; ${\cal L}C^1$-capacity; $C$-analytic capacity;
curvature of measure.

{\section{Introduction.}\label{s1}}

For the history of the subject under consideration we refer to the survey \cite{mpf2012}.
Let
$$
L({\bf x}) = c_{11}x_1^2 + 2c_{12}x_1x_2 + c_{22}x_2^2\;,\;\,{\bf x}=(x_1, x_2) \in \Rbb^2\,,
$$
be any fixed homogeneous polynomial of second order in $\Rbb^2$ (with constant {\it complex} coefficients $c_{11}, c_{12}, c_{22}$)
that satisfies the {\it ellipticity} condition:
$L({\bf x})\neq 0$ for all ${\bf x} \neq 0$. With the polynomial $L({\bf x})$ we associate the elliptic differential operator
$$
{\cal L} = c_{11}\frac{\partial^2}{\partial x_1^2} + 2c_{12}\frac{\partial^2}{\partial x_1 \partial x_2} + c_{22}\frac{\partial^2}{\partial x_2^2}\,.
$$
The basic examples: the Laplacian $\Delta$ and the Bitsadze operator
$\ov{\partial}^2 \equiv {\partial}^2/{\partial}\ov{z}^2$ in $\Rbb^2$
($z = x_1 + {\bf i} x_2$ is a complex variable).
For an open set $U$ in $\Rbb^2$ set $\,{\cal A}_{\cal L}(U) = \{\, u \in C^2(U)\,|\, {\cal L} u = 0\,$ in $\,U\}\,$.
The functions of this class we call ${\cal L}$-{\it analytic} in $U$. It is well known that
${\cal A}_{\cal L}(U)\subset C^{\infty}(U)$ (see \cite[Theorem 4.4.1]{Ho}).

Denote by $BC^1(U)$ ($U$ is open in $\Rbb^2$) the space of complex valued functions $f$ of class $C^1(U)$
with finite norm
$$
||f||_{1U}=\max \{||f||_U\,,\, ||\nabla f||_U\}\,,
$$
where $||g||_E = \sup_{{\bf x}\in E}|g({\bf x})|$ is the uniform norm of the (vector-) function $g$ on the set $E \neq \emptyset$
(for $U=\Rbb^2=E$ write $BC^1, ||f||_{1}$, $||g||$ respectively).

Let $X \neq \emptyset$ be a compact set in $\Rbb^2$ and $f\in BC^1$. The main problem considered in this paper consists
of the following.

{\it To find  conditions on ${\cal L}$, $X$ and $f$ necessary and sufficient for existence of a sequence
$\{f_n\}^{+\infty}_{n=1}\subset BC^1$ such that each $f_n$ is ${\cal L}$-analytic in (its own)
neighborhood of $X$ and $\|f-f_n\|_1 \to 0$ as $n\to+\infty$.}

The class of all functions $f\in BC^1$ with this approximation property is denoted by ${\cal A}^1_{\cal L}(X)$.
It is not difficult to show that the following condition always takes place: ${\cal A}^1_{\cal L}(X) \subset C^1_{\cal L}(X) = BC^1 \cap {\cal A}_{\cal L}(X^{\circ})$. Therefore, the following $C^1$-approximation problem {\it for classes of functions} naturally appears:

{\it For which compact sets $X$ one has ${\cal A}^1_{\cal L}(X) = C^1_{\cal L}(X)$? }

Recall that analogous approximation problems in the spaces $BC^m$ (see \cite{par2018} for the definitions) are solved for all $m > 0\,,\, m \neq 1$ (even in $\Rbb^N$, $N\geq 3$, \cite{far1}, \cite{ve2}, \cite{par2018}).

An open disc with center ${\bf a}$ and of radius $r > 0$ is denoted by $B({\bf a},r)\,$;
also, for $B=B({\bf a},r)$ and $\la>0$ by $\la B$ we mean the disc
$B({\bf a},\la r)$. Positive parameters (constants), which can depend only on ${\cal L}$, will be designated by
$A$, $A_1$, $A_2$, \dots (they may be different in different occurrences).

Let $\Phi({\bf x})=\Phi_{\cal L}({\bf x})$ be the standard fundamental solution for the equation ${\cal L}u=0$ (see \cite[p. 161]{ve2}
and Section 2 below). We shall basically use the so-called ${\cal L}C^1$-{\it capacity},
which is connected to the operator ${\cal L}$ and the space $BC^1$.
Namely, for a bounded nonempty set $E \subset \Rbb^2$ we let
\begin{equation}\label{kcap1}
\al_1(E) = \al_{1 \cal L}(E) =\sup_{T} \{| \langle T, 1\rangle| \,:\, {\Spt} (T) \subset E ,\, \Phi * T \in C^1(\Rbb^2)\,,\,||\nabla \Phi * T||\leq 1\}\,,
\end{equation}
where $\langle T, \f \rangle$ means the action of the distribution $T$ on the function
$\f \in C^{\infty}$, $*$ -- convolution, and ${\Spt} (T)$ is the support of the distribution (function, measure) $T$.

On the other hand, the ${\cal L}$-Lipschitz capacity of $E$, denoted by $\gamma_1(E)$ or
$\ga_{1 \cal L}(E)$ is defined in the same way as $\alpha_1(E)$, but with the condition $\Phi * T \in C^1(\Rbb^2)$
replaced by $\Phi * T \in {\rm Lip}(\Rbb^2)$.

Recall also the definition of the classic $C$-analytic capacity \cite{Vi}, basically used in the theory of uniform holomorphic approximations. For a bounded (nonempty) set $E \subset \Cbb$,
\begin{equation}\label{kcap}
\al(E) =\sup_{T} \{| \langle T, 1\rangle| \,:\, {\Spt} (T) \subset E ,\, 1/z * T \in C(\Cbb)\,,\,||1/z * T||\leq 1\}\,.
\end{equation}
If in this supremum one replaces the condition $1/z * T \in C(\Cbb)$ by $1/z * T \in L^\infty(\Cbb)$, one gets the well known analytic capacity, $\gamma(E)$, which is specially useful in the study of removable singularities for
bounded holomorphic functions \cite{Ahlfors}.

Clearly, for a bounded {\it open} set $E$ one has $\al_1(E)=\gamma_1(E)$ and $\al(E)=\gamma(E)$.

Our first main result is the following.

\begin{theor}\label{main-th1}
There exist constants $A_1 \in (0, 1)$ and $A_2 \geq 1$ dependent only on ${\cal L}$, such that
\begin{equation}\label{CapComp}
A_1\al(E) \leq \al_1(E) \leq A_2 \al(E)
\end{equation}
and
\begin{equation}\label{CapComp'}
A_1\ga(E) \leq \ga_1(E) \leq A_2 \ga(E)
\end{equation}
for any bounded set $E$.
\end{theor}

Let us remark that the capacities $\al$ and $\ga$ admit a characterization in terms of measures with linear growth and finite curvature, by \cite{tol04} and \cite{Tolsa-sem} (see \eqref{eqalph} and \eqref{eqgam}). This characterization extends to $\al_1$ and $\ga_1$ by the last theorem. In particular, the capacities $\al_1$ and $\ga_1$ are countably semiadditive.

For the above mentioned elliptic polynomial $\, L({\bf x})$, $\,f\in C(\Rbb^2)\,$ and a disc
$\,B=B({\bf a}, r)\,$ define the so-called $\,L$-{\it oscillation} of $\,f\,$ on $\,B\,$ (see \cite{par2018}):
$$
{\cal O}^L_{B} (f) = \frac{1}{2\pi r}\int_{\partial B} f({\bf x})
\frac{L({\bf x}-{\bf a})}{r^2} d\ell_{\bf x} - \frac{c_{11}+c_{22}}{2\pi r^2} \int_B f({\bf x})d{\bf x}\;,
$$
where $\ell$ is the Lebesgue measure (the length) on
$\partial B$.

For instance, when $\, L({\bf x})= x_{1}^2+x_2^2\,$ (that is ${\cal L}=\Delta$), we have (first time in analogous context
appeared in \cite{MPV}):
$$
{\cal O}^L_{B} (f) = \frac{1}{2\pi r}\int_{\partial B} f({\bf x})\,
d\ell_{\bf x} - \frac{1}{\pi r^2} \int_B f({\bf x})d{\bf x}\;,
$$
and for $\, L({\bf x}) = 4^{-1} (x_1 +  {\bf i} x_2)^2 = z^2/4\,$ (e.g. ${\cal L} = {\partial}^2/{\partial}\ov{z}^2$):
$$
{\cal O}^L_{B} (f) = \frac{1}{8\pi {\bf i} r^2}\int_{\partial B} f(z) (z-a)\,dz \;.
$$

Now we formulate our second main result.
Fix a compact set $X \subset \Rbb^2$ and $f\in BC^1$. Without loss of generality we shall suppose that
${\Spt} (f)$ is compact, e.g. $f \in C^1_0(\Rbb^2)$. Let $\om (g,r)$ be the modulus of continuity
of the (vector-)function $g$ on $\Rbb^2$.

\begin{theor}\label{main-th}
The following conditions are equivalent:

(a) $\, f \in {\cal A}^1_{\cal L}(X)\,$;

(b) there exist $k \ge 1$ and a function $\om (r)\to 0$ as $r\to 0+$ such that for each disk
$B=B({\bf a}, r)$ one has
$$
\left|{\cal O}^L_{B({\bf a}, r)} (f) \right| \le
\om (r)\, \al_1(B({\bf a}, kr)\setminus X)\;;
$$

(c) the property (b) holds for $k=1$ and $\om (r)=A\om (\nabla f, r)$.
\end{theor}

The plan of the paper is the following:

In Section 2 we give some preliminary results.

In Section 3 we prove Theorem \ref{main-th2} (more general than Theorem \ref{main-th}).

In Section 4 we prove Theorem \ref{main-th1} and recall the main properties of $C$-analytic capacity.

In Section 5 we present some corollaries of our main results: the $C^1$-approximation
criteria for classes of functions and corresponding criteria for the $C^1$-approximation by ${\cal L}$-polynomials
on plane compact sets.

{\section{Background.}

The next lemma is proved in \cite{par2018}.

\begin{lemma}\label{hat-osc}
For ${\bf a} \in \Rbb^2$ and $r\in (0, +\infty)$ let $\psi_r^{\bf a}({\bf x})=(r^2-|{\bf x}-{\bf a}|^2)/(4\pi r^2)$ in $B = B({\bf a}, r)$
and $\psi_r^{\bf a}({\bf x})=0$ outside of $B({\bf a}, r)$. Then  for all $\f \in C^{\infty} (\Rbb^2)$ one has:
$$
\int_{B} \psi_r^{\bf a} ({\bf x}) {\cal L}\f ({\bf x})\,d{\bf x} = {\cal O}^L_{B} (\f)\,,
$$
that is, the action $\langle {\cal L} \psi_r^{\bf a}\,, \, \f \rangle$ of the distribution ${\cal L} \psi_r^{\bf a}$ on the function
$\f$ is equal to ${\cal O}^L_{B} (\f)$ (and it can be continuously extended on all class of functions $\f \in C(\Rbb^2)$).
\end{lemma}

For a given $\f \in C_0^\infty(\Rbb^2)$ define the Vitushkin type localization operator (see \cite{Vi}, \cite{ve2})
corresponding to the operator ${\cal L}$:
\begin{equation}\label{eq3.0}
f \mapsto   {\cal V}_{\f}(f) =\Phi* (\f {\cal L} f) \;,\; f\in C(\Rbb^2)\,.
\end{equation}
The basic property of this operator consists of the following simple fact: ${\cal L} ({\cal V}_{\f}(f))=\f {\cal L} f$,
which means that ${\cal L}$-singularities of ${\cal V}_{\f}(f)$ are contained in the intersection of the ${\cal L}$-singularities of $f$ and $\Spt \f$.
We now present one new property of operator ${\cal V}_{\f}$ connected to the possibility of its extension to some wider class
of "indices" $\, \f$ whenever $f\in C^1$. For a compact set $X \subset \Rbb^2$ put $C^1(X)=C^1(\Rbb^2)|_X$.

\begin{lemma}\label{lemC1Vit}
Fix any function $\f \in C^1(\ov{B({\bf a},r)})$, $\f=0$ outside of $B({\bf a},r)$. Then for each $f \in C^1_0(\Rbb^2)$
the following properties hold:

(a) the function ${\cal V}_{\f}(f) \in C^1({\Rbb}^2)$ is well defined and
\[
   ||\nabla ({\cal V}_\f (f))||\leq A\,\om (\nabla f, r)\,||\nabla\f||\,r\,;
\]

(b) $\Spt ({\cal L}{\cal V}_{\f}(f)) \subset \Spt ({\cal L} f) \cap \Spt (\f)$;

(c) if $f \in C^2(\Rbb^2)$ then ${\cal L} ({\cal V}_{\f}(f))=\f {\cal L} f$.
\end{lemma}

\begin{proof}
We prove (a); (b) and (c) then follow from usual regularization arguments.
From the last mentioned arguments we also can additionally suppose that $f \in C_0^{\infty}(\Rbb^2)$.
Then it can be easily seen that ${\cal V}_\f (f) \in C^1({\Rbb}^2)$. Let $\om_1 (r)=\om(\nabla f, r)$.
{\it Fix} ${\bf x} \in \Rbb^2$ and for ${\bf y}\in \Rbb^2$ set
$$
   F({\bf y})= f({\bf y})-f({\bf x})-\nabla f({\bf x})\cdot({\bf y}-{\bf x})\;,\; \mbox{if} \;\, {\bf x} \in B({\bf a}, 2r)\,,
$$
$$
   F({\bf y})= f({\bf y})-f({\bf a})-\nabla f({\bf a})\cdot({\bf y}-{\bf a})\;,\; \mbox{if} \;\, {\bf x} \notin B({\bf a}, 2r)\,.
$$
Then ${\cal L} F = {\cal L} f$ and for all ${\bf y} \in \ov{B(a,r)}$ we have:
\begin{equation}\label{eq3.1}
|F({\bf y})|\leq 3\om_1(r) |{\bf y}-{\bf x}| \;, \; |\nabla F({\bf y})|\leq 3\om_1(r) \;, \; \mbox{if} \;\, {\bf x} \in B({\bf a}, 2r),
\end{equation}
\begin{equation}\label{eq3.1.1}
|F({\bf y})|\leq r\om_1(r) \;, \; |\nabla F({\bf y})|\leq \om_1(r) \;, \; \mbox{if} \;\, {\bf x} \notin B({\bf a}, 2r).
\end{equation}
Let $\pa_jg({\bf y})=\pa g({\bf y})/\pa y_j$ and $c_{21}=c_{12}$. Then
$$
\f {\cal L} F = {\cal L} (\f F) - F {\cal L} \f - 2 \sum_{i, j = 1}^2 c_{i j} \pa_{i}\f \pa_j F \,,
$$
so, substituting the last equality to \eqref{eq3.0} and taking into account that $\f ({\bf x}) F({\bf x})=0$, we obtain:
\begin{align*}
{\cal V}_\f (f)({\bf x}) & = \langle \Phi({\bf x}-{\bf y}), \f ({\bf y})\,{\cal L} F ({\bf y})\rangle
-\langle \Phi({\bf x}-{\bf y}), F ({\bf y})\, {\cal L} \f ({\bf y})\rangle - 2 \sum_{i, j = 1}^2 c_{i j} \langle \Phi({\bf x}-{\bf y}),
\pa_{i}\f ({\bf y}) \, \pa_j F ({\bf y})\rangle\,.
\end{align*}
From the equality $F\pa_i \pa_j\f = \pa_j (F\pa_i \f) - \pa_i \f \pa_j F$ we find that
\begin{align*}
-\langle \Phi({\bf x}-{\bf y}), F ({\bf y})\, {\cal L} \varphi ({\bf y})\rangle
= \sum_{i, j=1}^2 c_{ij}
\big(\langle \partial_j\Phi({\bf x}-{\bf y}), F ({\bf y})\,\partial_i \varphi ({\bf y})\rangle + \langle \Phi({\bf x}-{\bf y}),
   \pa_i\varphi ({\bf y})\,\pa_j F ({\bf y})\rangle \big)\,,
\end{align*}
which gives that
$$
{\cal V}_\f (f)({\bf x}) = \sum_{i, j=1}^2 c_{ij}
(\langle \pa_j\Phi({\bf x}-{\bf y}), F ({\bf y})\,\pa_i \f ({\bf y})\rangle - \langle \Phi({\bf x}-{\bf y}),
   \pa_i\f ({\bf y})\,\pa_j F ({\bf y})\rangle )\,,
$$
$$
\pa_k {\cal V}_\f (f)({\bf x}) = \sum_{i, j=1}^2 c_{ij}
(\langle \pa_k \pa_j\Phi({\bf x}-{\bf y}), F ({\bf y})\,\pa_i f ({\bf y}))\rangle - \langle \pa_k \Phi({\bf x}-{\bf y}),
   \pa_i\f ({\bf y})\,\pa_j F ({\bf y})\rangle )\,.
$$
From the last equality it follows that it remains to estimate the following expressions:
\begin{equation}\label{eq3.2}
\langle \pa_k\pa_j \Phi({\bf x}-{\bf y}), F ({\bf y})\,\pa_i\f ({\bf y})\rangle \,,
\end{equation}
and
\begin{equation}\label{eq3.3}
\langle \pa_k\Phi({\bf x}-{\bf y}), \pa_i\f ({\bf y})\,\pa_j F ({\bf y})\rangle
\end{equation}
for all possible triples $\{k, i, j\}$.

Since $|\nabla \Phi ({\bf y})| \leq A/|{\bf y}|$ (see Lemma \ref{solrep} below), by \eqref{eq3.1} and \eqref{eq3.1.1} the absolute value in \eqref{eq3.3} can be easily estimated by the following convergent integral:
$$
A_1||\nabla \f|| \,\om_1(r) \int_{B({\bf a},r)} \frac{d{\bf y}}{|{\bf x}-{\bf y}|}\,
\leq A\om_1(r) r ||\nabla \f||\,.
$$

It is not so directly simple to estimate \eqref{eq3.2}, because the kernel
$$
|\pa_k\pa_j \Phi ({\bf x}-{\bf y})| \simeq |{\bf x}-{\bf y}|^{-2}
$$
is not locally integrable with respect to the Lebesgue measure in $\Rbb^2$. Nevertheless, according to \cite[Lemma 1.1]{ve2}, for the function $\chi \in C_0(\Rbb^2)$ (properly tending to $0$ as ${\bf y} \to {\bf x}$) one has:
$$
\langle \pa_k\pa_j\Phi ({\bf x}-{\bf y}), \chi ({\bf y})\rangle = (v.p)\int {\cal K}_{kj}({\bf x}-{\bf y})\chi ({\bf y})\,d{\bf y}\,,
$$
where each $\,{\cal K}_{kj}\,$ is a standard (of class $C^{\infty}$ outside ${\bf 0}$, homogeneous of order $-2$, with zero
average over $\partial B({\bf 0}, 1)$) Calderon-Zygmund kernel in $\,\Rbb^2\,$ with respect to Lebesgue measure.
In our case, the function $\chi ({\bf y}) = F ({\bf y}) \pa_i \f ({\bf y})$ tends to zero like
$|{\bf y} - {\bf x}|$ as ${\bf y} \to {\bf x}$, because of \eqref{eq3.1} and since $\f=0$ outside of $B({\bf a},r)$.

Therefore, the last integral (in the principle value sense) as a matter of fact is absolutely convergent and can be estimated
for ${\bf x} \in B({\bf a}, 2r)$ as follows (using \eqref{eq3.1}):
\begin{multline*}
   \int\left | \pa_k\pa_j\Phi({\bf x}-{\bf y}) F({\bf y}) \pa_i
   \f ({\bf y})\right|\,d{\bf y} \leq  A_1\om_1(r) ||\nabla \f||\int_{B({\bf a},r)} \frac{d{\bf y}}{|{\bf x}-{\bf y}|}  \leq A \om_1(r) r ||\nabla \f||\,.
\end{multline*}
For ${\bf x} \notin B({\bf a}, 2r)$ the corresponding estimate is trivial (by \eqref{eq3.1.1}). Lemma \ref{lemC1Vit} is proved.
\end{proof}

Recall the basic properties for solutions of our main equation
$$
{\cal L}u = c_{11}\frac{\partial^2 u}{\partial x_1^2} + 2c_{12}\frac{\partial^2 u}{\partial x_1 \partial x_2} + c_{22}\frac{\partial^2 u}{\partial x_2^2} =0\,.
$$

Let $\l_1, \l_2$ be the roots of the {\it characteristic
equation}
$c_{11}\l^2 + 2c_{12}\l + c_{22} = 0$.
It follows from the ellipticity condition that
$\l_1, \l_2 \notin \Rbb$.
Define
 \begin{equation} \notag
 \begin{array}{lll}
 \pa_1 = \ddfrac{\pa}{\pa x_1} - \l_1\ddfrac{\pa}{\pa x_2},
 &\pa_2 = \ddfrac{\pa}{\pa x_1} - \l_2\ddfrac{\pa}{\pa x_2}
 &\quad \mbox{if $\l_1 \neq \l_2$,}\\
 \end{array}
 \end{equation}
 or
 \begin{equation} \notag
 \begin{array}{lll}
 \pa_1 = \ddfrac{\pa}{\pa x_1} - \l_1\ddfrac{\pa}{\pa x_2},
 &\pa_2 = \ddfrac{\pa}{\pa x_1} + \l_1\ddfrac{\pa}{\pa x_2}
 &\quad \mbox{if $\l_1 = \l_2$.}
 \end{array}
 \end{equation}

 \noindent
 We then have the following decomposition of ${\cal L}$:
 \begin{equation}\notag
 {\cal L}u =
 \begin{cases}
    c_{11}\pa_1(\pa_2(u)),   \qquad &\mbox{\rm if $\l_1 \neq \l_2$;} \\
    c_{11}\pa_1^2(u),   \qquad &\mbox{\rm if $\l_1 = \l_2$.}
 \end{cases}
 \end{equation}

 We also introduce the following new coordinates:
 \begin{equation} \notag
 \begin{array}{lll}
 z_1 = \ddfrac{\l_2}{\l_2-\l_1}(x_1 + \ddfrac{1}{\l_2}x_2),\;
 &z_2 = \ddfrac{\l_1}{\l_1-\l_2}(x_1 + \ddfrac{1}{\l_1}x_2) \quad
 &\quad \mbox{if $\l_1 \neq \l_2$;}\\
 \end{array}
 \end{equation}
 or
 \begin{equation}  \notag
 \begin{array}{lll}
 z_1 = \ddfrac{1}{2}(x_1 - \ddfrac{1}{\l_1}x_2),
 &z_2 = \ddfrac{1}{2}(x_1 + \ddfrac{1}{\l_1}x_2)
 &\quad \mbox{if $\l_1 = \l_2$.}
 \end{array}
 \end{equation}

 \noindent
 that satisfy the ``orthogonality" relations:
 $$
 \begin{array}{rlrrl}
 \pa_1 z_1 &= 1 & \quad &\pa_1 z_2 &= 0  \\
 \pa_2 z_1 &= 0 & \quad &\pa_2 z_2 &= 1 .
 \end{array}
 $$

 \noindent
 Finally, we "identify" $z=x_1+ {\bf i}x_2$ in $\Cbb$ and ${\bf x}=(x_1, x_2)$
 in $\Rbb^2$ in the sense that $f({\bf x})$ or $f(z)$ will mean the same for any function $f$. Notice that for $s=1$ and $s=2$ the linear transformations $\La_s(z) = z_s$ of $\Rbb^2$ are nondegenerate. Nevertheless, we shall never use ${\bf x}$ as a complex variable, as well as the symbol ${\bf x}_s$ (not the same as $x_s$) will not be used.

 \bigskip

 The following well known results take place
 \cite[Chapter IV, \S 6, (4.77)]{Bits}
 (see also \cite{PF99} for a simple direct proof).

 \bigskip
 \begin{lemma}\label{solrep}
 Let ${\cal L}$ be as above. If $\l_1 \neq \l_2$ then
 there exist in $\Cbb\setminus\{0\}$ a fixed analytic
 branch $\log(z_1 z_2^\nu)$ of the multivalued
 function $\Log(z_1 z_2^\nu)$ and a complex constant
 $k_1=k_1({\cal L}) \neq 0$ such that
 $$
 \Phi (z) =  \Phi_{\cal L}(z) = k_1\log(z_1 z_2^\nu)
 $$
 is a fundamental solution of ${\cal L}$, where $\nu = 1$
 if $\operatorname{sgn}(\Im \l_1) \neq \operatorname{sgn}(\Im \l_2)$, and $\nu = -1$ otherwise.

 If $\l_1 = \l_2$, then $\Phi_{\cal L}(z) = k_1\ddfrac{z_1}{z_2}$
 is a fundamental solution of ${\cal L}$,
 where $k_1=k_1({\cal L}) \neq 0$.

 \end{lemma}

\begin{lemma}\label{solpot}
There is $k_2=k_2({\cal L}) > 1$ with the following properties. Let $T$ be a distribution with compact support in the disc $B(a,r)$
and $g=\Phi_{\cal L}*T$.

If $\l_1\ne\l_2$ then for $|z-a|>k_2 r$ we have the expansion
\begin{equation}\label{eq3.4}
g(z)=c_0\Phi(z-a) + \sum_{m=1}^{\infty}\frac{c_m^1}{(z-a)_1^m} +
\sum_{m=1}^{\infty}\frac{c_m^2}{(z-a)_2^m},
\end{equation}
where $c_0 = c_0(g)= \langle T, 1 \rangle$ and
$$
c_m^s=c_m^s(g, a)= -k_1\dfrac{\nu^{s-1}}{m} \langle T , (w-a)_s^m \rangle\;,\; s \in \{1, 2\}\,,\; m=1,2, \dots .
$$

If $\l_1 = \l_2$ then for $|z-a|>k_2 r$ we have the expansion
\begin{equation}\label{eq3.4.1}
g(z)=c_0\Phi(z-a) + \sum_{m=1}^{\infty}\frac{c_m^1}{(z-a)_2^m} +
\sum_{m=1}^{\infty}\frac{c_m^2 (z-a)_1}{(z-a)_2^{m+1}},
\end{equation}
where $c_0 = c_0(g)= \langle T, 1 \rangle$ and
$$
c_m^1 = -k_1 \langle T(w) , (w-a)_1(w-a)_2^{m-1} \rangle\;,
$$
$$
c_m^2 = k_1 \langle T(w) , (w-a)_2^m \rangle\;,\; m=1,2, \dots\,.
$$
The series in \eqref{eq3.4} and \eqref{eq3.4.1} converge in  $C^{\infty}(\Cbb\setminus B(a,k_2 r))$.

\end{lemma}

\begin{example}
         For the Laplacian ${\cal L} = \D$, one has $\l_1={\bf i}$, $\l_2=-{\bf i}$,
         $z_1=z/2$, $z_2=\bar{z}/2$ and
         $$
          \pa_1=\ddfrac{\pa}{\pa x_1} - {\bf i}\ddfrac{\pa}{\pa x_2}
          =: 2\ddfrac{\pa}{\pa z}, \quad
          \pa_2=\ddfrac{\pa}{\pa x_1} + {\bf i}\ddfrac{\pa}{\pa x_2}
          =: 2\ddfrac{\pa}{\pa \bar{z}}, \quad
          \Phi_{\D}(z) = \ddfrac{1}{4\pi}\log(\ddfrac{z \bar{z}}{4}).
         $$

         For the Bitsadze operator
         ${\cal L} = \ddfrac{\pa^2}{\pa {\bar{z}}^2 } =
         \ddfrac{1}{4}\left(\ddfrac{\pa^2}{\pa x_1^2} +
            2{\bf i}\ddfrac{\pa^2}{\pa x_1 \pa x_2} -
            \ddfrac{\pa^2}{\pa x_2^2}\right)$,
         one gets $\l_1=\l_2=-{\bf i}$,  $z_1=\bar{z}/2$, $z_2=z/2$ and
         $$
          \pa_1= 2\ddfrac{\pa}{\pa \bar{z}}, \quad
          \pa_2= 2\ddfrac{\pa}{\pa z}, \quad
          \Phi_{\cal L}(z) = \ddfrac{1}{\pi}\ddfrac{\bar{z}}{z}.
         $$
\end{example}

For a class $\cal I$ of functions and $\tau\geq 0$ we denote by
$\tau \cal I$ the class $\{\tau g:\, g\in \cal I\}$. Rewrite the definition of $\al_{1 \cal L}(E)$ for a nonempty bounded set $E$:
$$
\al_1(E) = \al_{1 \cal L}(E) =\sup \{| \langle {\cal L}g, 1\rangle| \,:\, g \subset {\cal I}_1(E)\} \,,
$$
where
$$
{\cal I}_1(E)=\{ \Phi_{\cal L} * T\,|\, \Spt (T) \subset E\,, \Phi_{\cal L} * T \in C^1(\Rbb^2)\,,\,||\nabla \Phi * T||\leq 1\}\,.
$$

Clearly, $\al_1(B(a, r)) \leq A r$ for each disc $B(a, r)$, and this is the only property of $\al_1$ that we
need in this and the next section.

For $g \in C^1(\Rbb^2)$ define $\nabla^c g =(\pa_1 g, \pa_2 g)$.
Then $|\nabla^c g|$ is comparable to $|\nabla g|$ and $\om (\nabla^c g, r)$ is comparable to $\om (\nabla g, r)$.

The following lemma (where we use the notations of Lemma \ref{solpot} above) is analogous to Lemma 3.3 and Corollary 3.4 in \cite{P}.
\begin{lemma}\label{estadm}
Let $E\subset B(a,r)$ and $g\in {\cal I}_1(E)$. Then there are  $k_3=k_3({\cal L})>1, k_4=k_4({\cal L})>1$ and $A=A({\cal L})>0$ such that
\begin{equation}\label{eq3.4.4}
|c_0(g)|\leq \al_1(E) \;,\;  |c_m^s(g,a)|\leq A(k_3r)^m\al_1(E)\;, s \in \{1, 2\}\,,\; m=1,2, \dots .
\end{equation}
and for $|z-a|>k_4r$ one has
\begin{equation}\label{eq3.4.51}
|\nabla^c g(z)| \leq\frac{A\al_1(E)}{|z-a|}\;;
\end{equation}
\begin{equation}\label{eq3.4.52}
\left|\nabla^c (g(z)- c_0\Phi_{\cal L}(z)) \right| \leq \frac{A r\al_1(E)}{|z-a|^2}\;;
\end{equation}
$$
\left|\nabla^c \left(g(z)-c_0\Phi_{\cal L}(z) -  \frac{c_1^1}{(z-a)_1} -
\frac{c_1^2}{(z-a)_2}\right)\right|\leq \frac{Ar^3}{|z-a|^3}\;,\; \mbox{if}\; \l_1\ne\l_2\;;
$$
\begin{equation}\label{eq3.4.53}
\left|\nabla^c \left(g(z)-c_0\Phi_{\cal L}(z) -  \frac{c_1^1}{(z-a)_2} -
\frac{c_1^2 (z-a)_1}{(z-a)_2^2}\right)\right|\leq \frac{Ar^3}{|z-a|^3}\;,\; \mbox{if} \; \l_1=\l_2\,.
\end{equation}
\end{lemma}
\begin{proof}
We give a short proof here for completeness, and only for the cases $\l_1\ne\l_2$. The cases $\l_1 =\l_2$
can be done almost the same way. To check \eqref{eq3.4.4} fix $m\geq 1$ and $s$ and take $g_m^s=\Phi*(T(w)(w-a)_s^m)$,
where $T={\cal L}g$. Then $|c_m^s|=|k_1|m^{-1}|c_0(g_m^s)|$. Take $\chi\in C^1(\ov{B({\bf a},2r)})$, $\chi=0$
outside of $B({\bf a},2r)$, with $\chi=1$ in $B({\bf a},r)$ and $||\nabla \chi||<2/r$. Then $g_m^s= {\cal V}_{\chi(w)(w-a)_s^m}g$.
Since, clearly, $||\nabla (\chi(w)(w-a)_s^m)||\leq A(k_3r)^{m-1}$, it remains apply Lemma \ref{lemC1Vit}
and definition of $\al_1(E)$. The remaining estimates in the last lemma can be now easily checked.
In fact, let $d_1=\min\{|z_s|\;:\, |z|=1\,,\, s \in \{1, 2\}\}$. One then can take $k_4=(k_3+1)/d_1$.
\end{proof}

{\section{Proof of Theorem \ref{main-th}.}}

We now formulate and prove some generalization of Theorem \ref{main-th}.

Fix any {\it even} function $\f_1$ in $C(\Rbb^2)\cap C^1(\ov{B({\bf 0},1)})$ with
$\Spt\f_1$ in $\ov{B({\bf 0},1)}$ and with the property $\int\f_1({\bf x})d{\bf x} =1$.

Set $\f_r^{\bf a}({\bf x}) =\f_1(({\bf x}-{\bf a})/r)/r^2\,$ and $\varphi_r =\varphi_r^{\bf 0}$. Clearly,
$||\nabla \f_r^{\bf a}||=r^{-3}||\nabla \f_1||$.

By analogy with \cite[Theorem 2.2]{P95}, the proof of Theorem \ref{main-th} is based on the
following result.

\begin{theor}\label{main-th2}
For a compact set $X$ and $f\in C^1_0(\Rbb^2)$ the following are equivalent:

(a) $\, f \in {\cal A}^1_{\cal L}(X)\,$;

(b) there exist $k \ge 1$ and a function $\om (r)\to 0$ as $r\to 0+$ such that for each disc $B=B({\bf a}, r)$ one has
\begin{equation}\label{eq3.5}
 \left|\int_{B({\bf a}, r)} \pa_1 f ({\bf x}) \pa_2\f_r^{\bf a} ({\bf x}) d {\bf x} \right| \le
\om (r) r^{-2} \,\al_1 \left(B({\bf a}, kr)\setminus X \right)\;;
\end{equation}

(c) the property (b) holds for $k=1$ and $\om (r)=A\om (\nabla f, r)$.
\end{theor}

In particular, for $\f_1 = 8\psi_1^{\bf 0}$ (see Lemma \ref{hat-osc}) this theorem
coincides with Theorem \ref{main-th}.

{\it Proof of $(a)\Rightarrow (c)$ in Theorem \ref{main-th2}}. Let
$f \in {\cal A}^1_{\cal L}(X)$ and take a sequence $\{f_n\}^{+\infty}_{n=1}\subset BC^1$ such that
each $f_n$ is ${\cal L}$-analytic in (its own)
neighborhood $U_n$ of $X$ and $\|f-f_n\|_1 \to 0$ as $n\to+\infty$. By regularization arguments we can additionally suppose that each $f_n \in C^{\infty} (\Rbb^2)$.
Fix $B=B({\bf a}, r)$ and $\e \in (0, r/2)$. Then there is $n_{\e}\in \Nbb$ such that for all $n \geq n_{\e}$ one has $\|f-f_n\|_1 < \e$, and then also $\om ((\na f - \na f_n), r)< 2\e$.
So it is enough to prove the estimate
$$
 \left|\int_{B({\bf a}, r)} \pa_1 f_n ({\bf x}) \pa_2\f_r^{\bf a} ({\bf x}) d {\bf x} \right| \le
A \om_n (r) r^{-2} \,\al_1 \left(B({\bf a}, r)\setminus X \right)
$$
with $A=A({\cal L})$ and $\om_n (r)=\om (\na f_n, r)$, and then tend $\e$ to $0$.
Let $h_n={\cal V}_{\f}f_n$, where $\f({\bf x})= \f_{\bf a}^{r-\e}({\bf x})$.
By Lemma \ref{lemC1Vit}, $h_n \in BC^1$, $||\nabla h_n|| \leq A\om_n (r) r||\nabla \f||$ and $h_n$ is ${\cal L}$-analytic outside some
compact set $E \subset B\setminus X$. By \eqref{kcap1} and Lemma \ref{lemC1Vit} we have
$$
|\langle {\cal L} h_n, 1  \rangle|=|\langle \f, {\cal L} f_n \rangle|=|c_{11}||\langle \pa_2 \f , \pa_1 f_n  \rangle| \le
A\om_n (r) r||\nabla \f|| \al_1 \left(B({\bf a}, r)\setminus X \right)\,,
$$
which ends the proof of $(a)\Rightarrow (c)$.

Since $(c)\Rightarrow (b)$ is evident, we pass to the following more complicated part of the proof.

{\it Proof of $(b)\Rightarrow (a)$ in Theorem \ref{main-th2}}.

We can suppose that for some $R>0$ we have
$X\subset B(0,R)$ and $f(z)=0$ for $|z|>R$. In \eqref{eq3.5} we also take
$\omega(\delta)\geq \omega(\nabla f,\delta)$.

Fix $\delta>0$ and any standard $\delta$-partition of unity
$\{(\varphi_j,B_j)\,:\; j=(j_1,j_2)\in\Zbb^2\}$ in $\Cbb$. This means that
$B_j = B(a_j,\delta)$, where $a_j=j_1\delta + {\bf i} j_2\delta\in\Cbb$,
$\varphi_j\in C_0^\infty(B_j)$, $0\leq \varphi_j\leq
1$, $||\nabla\varphi_j||\leq A/\delta$, $\sum_j \varphi_j\equiv 1$.

Now consider the new partition of unity $\{(\psi_j, B'_j)\}$, where
$\psi_j=\varphi_\delta*\varphi_\delta*\varphi_j$,
$B'_j=B(a_j,3\delta)$ (recall that
$\varphi_\delta=\varphi_\delta^{\bf 0}$). Clearly, $\psi_j\in
C_0^\infty(B'_j)$ and $||\nabla\psi_j||\leq A/\delta$.
Define the so called {\it localized} functions $f_j = \Phi_{\cal L}*(\psi_j {\cal L}f)$.

\begin{lemma}\label{LocEst}
The functions $f_j$ satisfy the following properties:
\begin{itemize}
\item[(1)] $f_j\in A\omega(\nabla f,\delta){\cal I}_1(B'_j\setminus X^0)$;
\item[(2)] $f=\sum_j f_j$ and the sum is finite ($f_j=0$ if $B'_j\cap
B(0,R) =\emptyset$);
\item[(3)] if $\l_1\ne\l_2$ then for $|z-a_j|>3k_2 \d$ we have the expansion
$$
f_j(z)=c_{0j}\Phi(z-a_j) + \sum_{m=1}^{\infty}\frac{c_{mj}^1}{(z-a_j)_1^m} +
\sum_{m=1}^{\infty}\frac{c_{mj}^2}{(z-a_j)_2^m}\,,
$$
or for $\l_1=\l_2$:
$$
f_j(z)=c_{0j}\Phi(z-a_j) + \sum_{m=1}^{\infty}\frac{c_{mj}^1}{(z-a_j)_2^m} +
\sum_{m=1}^{\infty}\frac{c_{mj}^2 (z-a_j)_1}{(z-a_j)_2^{m+1}}\,,
$$
where
\begin{equation}\label{eq3.5.2}
c_{0j} = \int f({\bf x}) {\cal L}\psi_j ({\bf x}) d{\bf x}=-c_{11}\int \pa_1 f({\bf x}) \pa_2 \psi_j ({\bf x}) d{\bf x}\,,
\end{equation}
\begin{equation}\label{eq3.5.3}
c_{1j}^s=c_{11}k_1\nu^{s-1} \int \pa_s f({\bf x}) \pa_{(3-s)}(\psi_j ({\bf x}) (z-a_j)_s)d{\bf x}\;,\; s \in \{1, 2\}
\end{equation}
with $\nu=-1$ whenever $\l_1=\l_2$.
Define $G_j=B(a_j, (k+2)\delta)\setminus X$. Then
\begin{equation}\label{eq3.5.4}
|c_{0j}| \leq A\omega(\nabla f,\delta) \al_1(G_j)\,,
\end{equation}
\begin{equation}\label{eq3.5.5}
|c_{1j}^s| \leq A\omega(\nabla f,\delta) \delta \al_1(G_j)\,\,, s \in \{1, 2\}\,.
\end{equation}

\end{itemize}
\end{lemma}

\begin{proof}
Notice that the last two estimates are corollaries of \eqref{eq3.5}, not only \eqref{eq3.4.4}.
We follow analogous proof for Lemma 2.5 in \cite{P95}.

First we obtain \eqref{eq3.5.4} using \eqref{eq3.5.2}, which follows from Lemma \ref{solpot}, the definition of
$f_j$ and integration by parts. Set $\f_j^* =\f_\delta*\f_j$.
Then
$$
\varphi_j^* \in C_0^\infty(B(a_j,2\delta))\;,\; 0\leq\varphi_j^*\leq 1 \;,\; \psi_j =\f_\delta*\f_j^*\,.
$$

By \eqref{eq3.5} and Fubini's theorem
\begin{align*}
|c_{0j}| & = |c_{11}| \left|\int \pa_1 f({\bf x}) \pa_2 \left(\int\varphi_\delta({\bf x}-{\bf y})\varphi_j^*({\bf y}) d{\bf y} \right) d{\bf x}\right| \\
& = |c_{11}| \left| \int\varphi_j^*({\bf y})
\left(\int \pa_1f({\bf x})\pa_2(\varphi_\delta^{{\bf y}}({\bf x})) d{\bf x}
\right)d{\bf y} \right| \\
&
\leq A_1 \left| \int\varphi_j^*({\bf y})\omega(\delta)\delta^{-2}\al_1(B({\bf y}, k\delta)
\setminus X) d{\bf y} \right| \leq  A\omega(\delta)\alpha_1(G_j)\,.
\end{align*}

In order to estimate $|c_{1j}^s|$ we first need to
check that in \eqref{eq3.5.3} the function $\psi_j ({\bf x}) (z-a_j)_s$
has the form $\varphi_\delta*\chi_j$, where
$\chi_j\in C_0^\infty(B(a_j, 2\delta))$ and $||\chi_j||\leq A\delta$.
It can be done the same way as in \cite[p. 1331]{P95} or \cite[Lemma 3.4]{mazpar15} using the Fourier transform.
Then we proceed as in the first part of the proof.

\end{proof}

We are ready to describe the scheme for approximation of the
function $f=\sum f_j$ following \cite{{P}} and \cite[\S 6]{mazpar15}.

Put $J=\{j\in \Zbb^2\,:\; B'_j\cap\partial X\neq \emptyset\}$. For $j\not\in J$
by Lemma \ref{LocEst} (1), clearly, $f_j\in {\cal A}^1_{\cal L}(X)$, so these $f_j$ don't need to be
approximated. Let now $j\in J$. By definition of $\alpha_1(G_j)$ (recall that $G_j=B(a_j, (k+2)\delta)\setminus X$)
and by \eqref{eq3.5.4}
we can find functions $f_j^*\in A\omega(\delta){\cal I}_1(G_j)\subset {\cal A}^1_{\cal L}(X)$
such that $c_0(f_j^*)=c_0(f_j)$. Put $g_j=f_j-f_j^*$ ($f_j^*= f_j$,
$g_j\equiv 0$ for $j\not\in J$). Then
\begin{equation}\label{eq3.5.6}
    ||\nabla g_j|| \leq A\omega(\delta)\;;\; c_0(g_j)=0\, .
\end{equation}
Therefore, by \eqref{eq3.4.4} (with $m=1$) for $E=G_j$ and $g=f_j^*$ and by \eqref{eq3.5.5} we can write
(clearly, $c_1^s(g_j, a)=c_1^s(g_j)$ do not depend on $a$):
\begin{equation}\label{eq3.5.7}
   |c_1^s(g_j)| \leq A\omega(\delta)\delta\alpha_1(G_j)\,, s \in \{1, 2\}.
\end{equation}

Using \eqref{eq3.4.51} -- \eqref{eq3.4.53} for $g=g_j$ and $E=B(a_j, (k+2)\delta)=B^*_j$, we
obtain for $|z-a_j|>p\delta$ (here $p=\max\{k_2, k_3, k_4, k+2\}+1$) :
\begin{equation}\label{eq3.5.8}
    |\nabla g_j(z)|\leq \frac{A\omega(\delta)\delta\alpha_1(G_j)}
                      {|z-a_j|^2} + \frac{A\omega(\delta)\delta^3}
                      {|z-a_j|^3}\,.
\end{equation}

We need to introduce the following abbreviate notations. Recall that $\delta$ is fixed and small enough.

For $j \in J$ set $\al_j=\al_1(G_j)$, so that all $\al_j>0$. For $I \subset J$ and $z\in \Cbb$
put
$$
B^*_I=\bigcup_{j\in I}B^*_j\;,\; G_I=\bigcup_{j\in I}G_j\;,\; \al_I=\sum_{j\in I} \al_j\;,\; g_I=\sum_{j\in I} g_j\;,
$$
$$
I'(z)=\{j \in I\;:\; |z-a_j|>p\d\}\;,\; S'_I(z)=\sum_{j\in I'(z)} \left(\frac{\d \al_j}{|z-a_j|^2}+\frac{\delta^3}
                      {|z-a_j|^3}\right)\,.
$$
Set also $S_I(z)=S'_I(z)$ if $I=I'(z)$ and $S_I(z)=S'_I(z)+1$ if $I\neq I'(z)$.
For $I\subset J$, $l\in I$ and $s\in \{1, 2\}$ define $P_s(I,l)=\{j\in I\,:\, j_{3-s}=l_{3-s}\}$.

\begin{defin}\label{def1} Fix $s\in \{1, 2\}$, $I\subset J$ and $l\in I$. A subset $L_s=L_s(l)$ of $I$ is called a {\it complete}
$s$-{\it chain} in $I$ {\it with vertex} $l$ if the following conditions are satisfied:
\begin{itemize}
\item[(1)] $L_s$ is $s$-{\it directional and connected in} $I$; this
means that $L_s\subset P_s(I,l)$, $j_s\geq l_s$ for all $j \in L_s$,
and for each $j\in L_s$ and $j'\in P_s(I,l)$ such that $l_s \leq j'_s \leq j_s$
we have $j'\in L_s$;

\item[(2)] it is possible to represent $L_s$ as
$L_s=L_s^1\cup L_s^2 \cup L_s^3$ with the following
properties: for each $j^{\theta}\in L_s^{\theta}$, $\theta=1,2,3$, one has
$$
   j_s^1<j_s^2<j_s^3\;\quad \mbox{and}\quad|a_{j^1}-a_{j^3}|\geq q\delta \;,
$$
where $q\geq 3p$ depending only on ${\cal L}$ will be chosen later;

\item[(3)] for $\theta=1$ and $\theta=3$ we have $\alpha_{L_s^{\theta}} \geq \delta $ and $L_s$
is minimal with the properties above (then, clearly, $\alpha_{L_s} \leq A\delta$).

\end{itemize}
\end{defin}

\begin{defin}
Let $l \in I \subset J$. A set $\Ga \subset I$ is called a complete group in $I$ with vertex $l$
if there exist complete $1$- and $2$-chains $L_1$ and $L_2$  in $I$ with vertex $l$ such that
$\Ga=L_1\cup L_2$.
\end{defin}

Now we divide the set of indices $J$ into a finite number of
nonintersecting groups $\Gamma^n$, $n \in \{1, \dots, N\}$ by
induction as follows. First define a natural order in $J$: for $j \neq j'$ in $J$
write $j<j'$ if $j_2<j'_2$ or $j_2=j'_2$ but $j_1<j'_1$. Now choose
the minimal $l^1$ in $J$. If there exists a complete group $\Ga=L_1\cup L_2$ in $J$ with vertex $l^1$
 we define $\Ga^1=\Ga$. If such $\Ga$ does not exist, we put $\Ga^1=P_1(J,l)$
if $L_1$ does not exist and call $\Ga^1$ {\it incomplete} $1$-group, otherwise put $\Ga^1=P_2(J,l)$ (if $L^1$ exists, but $L_2$ does not exist in the above sense) and call $\Ga^1$ {\it incomplete} $2$-group.
If $\Ga^1, \dots, \Ga^n$ are constructed, take $J^{n +1}=J\setminus (\Ga^1\cup \dots \cup \Ga^n)$
and make the same procedure for $J^{n +1}$ instead of $J$ defining $\Ga^{n+1}$. Let $N$ be the maximal
number with the property $J^N \neq \emptyset$.
Now we fix this partition $\{\Gamma^n\}=\{\Gamma^n\}_{n=1}^N$ of $J$.

For each group $\Gamma= \Gamma^n$ (complete or not) by \eqref{eq3.5.6} -- \eqref{eq3.5.8} one has:
$$
\alpha_\Gamma\leq A\delta\,, \, c_0(g_\Gamma)=0\,,\,
|c_1^s(g_\Gamma)|\leq A\omega(\delta)\delta \,(s\in \{1, 2\})\,;
$$
\begin{equation}\label{eq3.5.10}
|\nabla g_\Gamma(z)|\leq A\omega(\delta)S_{\Gamma}(z)\,, \,
||\nabla g_{\Gamma}|| \leq A\omega(\delta)\,, ||S_{\Ga}|| \leq A\,,\, ||S_{P_s(J, l)}|| \leq A\, (s\in \{1, 2\}).
\end{equation}

\begin{lemma}\label{main}
For each complete group $\Gamma=\Gamma^n$
there exists $h_\Gamma\in A\omega(\delta){\cal I}_1(G_\Gamma)\subset
{\cal A}^1_{\cal L}(X)$ such that
$$
c_0(h_\Gamma)=0\;, \; c_1^s(h_\Gamma)=c_1^s(g_\Gamma)\; (s\in \{1, 2\})\,,
$$
and for all $z\in\Cbb$
$$
  |\nabla h_\Gamma(z)|\leq A\omega(\delta)S_\Gamma(z)\,.
$$
\end{lemma}

\begin{proof}
We follow the idea in \cite[Lemma 2.7]{P95}.
Let $\Gamma$ be a complete group in $J$ with vertex $l$ and complete
$1$- and $2$-chains $L_1$ and $L_2$ respectively, and let
$L_1=L_1^1\cup L_1^2 \cup L_1^3$ (like in definitions just above).

For each $j \in \Gamma$ we can choose $h_j\in 2 {\cal I}_1(G_j)$ with $c_0(h_j)=\al_j = \al_1(G_j)$.
Let $T_j={\cal L}h_j$, so that $\al_j=\langle T_j \,,\, 1\rangle$.
Fix $j^1 \in L^1_1$, $j^3 \in L^3_1$ and for $\theta =1$ and $\theta =3$ put:
$$
 h^\theta=h_{j^\theta} \,, \; T^\theta ={\cal L} h_{j^\theta} \,, \;
a^{\theta}=a_{j^\theta} \,, \; G^\theta=G_{j^\theta} \,.
$$
Put $M=|a^1-a^3|/\delta$. Let $\la^1 \in (0,1)$ and $\la^3 \in (0,1)$ be such that $\la^1\al^1=\la^3 \al^3:=\al$. Define
\begin{equation}\label{eq3.5.11}
    h^{13}(z)=h^{13}(j^1,j^3,\la^1,\la^3, z)= (\la^3h^3(z)-\la^1 h^1(z))/M\,.
\end{equation}
Then, clearly, $c_0(h^{13})=0$ and by Lemma \ref{solpot}
\begin{align*}
\nu^s k_1^{-1} M c_1^s(h^{13})& =\langle \la^3 T^3 - \la^1 T^1 \,,\, z_s \rangle \\
&=
\la^3\langle T^3 \,,\, (z-a^3)_s \rangle + \la^3\langle T^3 \,,\,a^3_s \rangle -
 \la^1\langle T^1\,,\, (z-a^1)_s \rangle -
\la^1\langle T^1 \,,\, a^1_s \rangle\\ &= \al (a^3 -a^1)_s + R^{13}_s\,,
\end{align*}
where $|R^{13}_s| \leq A\d \al$, which follows from Lemma \ref{solpot} and \eqref{eq3.4.4} with $m=1$.
Therefore,
\begin{align}\label{eq3.5.12}
(c_1^1(h^{13})\,,\; c_1^2(h^{13}))&=k_1 M^{-1}\al(\nu (a^3 -a^1)_1, (a^3 -a^1)_2) + \d \al O(1/M)\nonumber\\
& =k_1\d \al\left((\nu (1+{\bf i}0)_1, (1+{\bf i}0)_2) + O(1/M)\right)\,.
\end{align}
Moreover, for $z$ with $|z-a^1|>p\d$ and $|z-a^3|>p\d$ we have by \eqref{eq3.4.52}:
\begin{align}\label{eq3.5.13}
\nabla h^{13}(z) &=\al M^{-1}(\nabla \Phi(z-a^3)-\nabla \Phi(z-a^1)) + \al \d (O(|z-a^1|^{-2}) + O(|z-a^3|^{-2}))\nonumber \\
&=\al \d (O(|z-a^1|^{-2}) + O(|z-a^3|^{-2}))\,.
\end{align}
The last equality has to be checked: it is enough, instead of estimating $|\nabla \Phi(z-a^1)-\nabla \Phi(z-a^3)|$,
to estimate $|\pa_s \Phi(z-a^1)-\pa_s \Phi(z-a^3)|$, $s \in \{1, 2\}$ (see Lemma \ref{solrep}),
using, finally, the following elementary trick:
$$
\left|\frac{1}{(z-a^1)_s} -\frac{1}{(z-a^3)_s}\right|=\left|\frac{(a^3-a^1)_s}{(z-a^1)_s(z-a^3)_s}\right|\leq
A|a^3-a^1|(|z-a^1|^{-2} + |z-a^3|^{-2})\,.
$$
In fact the last can be done for the case $\la_1\neq \la_2$ (see Lemma \ref{solrep}). The remaining case ($\la_1 = \la_2$) we propose to the reader's control. Additionally notice that for the case when $|z-a^{\theta}|\leq p\d$, $\theta =1$ or $\theta =3$, we have
by \eqref{eq3.4.51} (since $|a^3-a^1|=M\d \geq q\d \geq 3p\d$):
\begin{equation}\label{eq3.5.14}
|\nabla h^{13}(z)| \leq \frac{\la^{\theta}}{p}+A\frac{\la^{4-\theta} \al^{4-\theta}\d}{|z-a^{4-\theta}|^2}\,.
\end{equation}

Now we construct a special linear combination (just a sum) of such functions
$h^{13}(z)=h^{13}(j^1,j^3,\lambda^1,\lambda^3, z)$. It is easily seen
that for each $j\in L_1^1\cup L_1^3$ there exist $\lambda(j, \ka)$
$(\ka \in \{1,\dots, \ka_j\}\,, \ka_j \in \Nbb)$ with the following properties:
\begin{itemize}

\item[(a)] $\lambda(j, \ka)> 0$, $\sum_{\ka=1}^{\ka_j}\lambda(j,\ka) \leq
1$ for each $j$;

\item[(b)] between the sets of indices
$$
   \Psi^\theta=\{(j,\ka)\,:\; j\in L_1^\theta\,,\; 1\leq \ka\leq \ka_j\}\,,\;
   \theta=1\;\mbox{and}\; 3\,,
$$
we have one to one correspondence
$$
  \Psi^1\ni (j^1,\ka^1)\longleftrightarrow (j^3,\ka^3)\in \Psi^3\,,
$$
for which $\lambda(j^1,\ka^1)\alpha_{j^1}=\lambda(j^3,\ka^3)\alpha_{j^3}$;

\item[(c)] for $\theta=1$ and $\theta=3$
$$
    \sum_{(j,\ka)\in\Psi^\theta}\lambda(j,\ka)\alpha_j=\delta\,.
$$

\end{itemize}
We define
$$
    h_1(z)=h_1(L_1, z)=\sum_{(j^1,\ka^1)\in\Psi^1} \frac{\delta}{|a_{j^3}-a_{j^1}|}
       (\lambda(j^3,\ka^3)h_{j^3} -\lambda(j^1,\ka^1)h_{j^1})\,,
$$
where $(j^3,\ka^3)$ corresponds to $(j^1,\ka^1)$ in the above sense. Each
member of the last sum is precisely of the form \eqref{eq3.5.11} with
$\lambda^\theta=\lambda(j^\theta, \kappa^\theta)$. Clearly, $c_0(h_1)=0$ and
by \eqref{eq3.5.12} we have
\begin{equation}\label{eq3.5.15}
(c_1^1(h_1)\,,\; c_1^2(h_1))=k_1\d^2\left((\nu (1+{\bf i}0)_1, (1+{\bf i}0)_2) + O(1/M)\right)\,.
\end{equation}

Arguing the same way for the complete chain $L_2$ of $\Ga$ we construct the function
$h_2(z)=h_2(L_2, z)$ with the same properties as for $h_1$, but
\begin{equation}\label{eq3.5.16}
(c_1^1(h_2)\,,\; c_1^2(h_2))=k_1\d^2\left((\nu (0+{\bf i}1)_1, (0+{\bf i}1)_2) + O(1/M)\right)\,.
\end{equation}
Now {\it choose and fix} $q$ so large in Definition \ref{def1} that $O(1/M)=O(1/q)$ does not "spoil" (in \eqref{eq3.5.15}
and \eqref{eq3.5.16}) the linear independence of the vectors $k_1\d^2((\nu (1+{\bf i}0)_1, (1+{\bf i}0)_2))$ and
$k_1\d^2((\nu (0+{\bf i}1)_1, (0+{\bf i}1)_2))$. Notice that (in $\Cbb^2$) the vector $(\nu (1+{\bf i}0)_1, (1+{\bf i}0)_2)$
is collinear to $(\nu \la_2, -\la_1)$ (when $\la_1\neq \la_2$) and to $(-1,1)$ (if $\la_1 = \la_2$); the vector
$(\nu (0+{\bf i}1)_1, (0+{\bf i}1)_2)$ is collinear to $(\nu, -1)$ (when $\la_1\neq \la_2$) and to $(1,1)$ (if $\la_1 = \la_2$).

By \eqref{eq3.5.13}, \eqref{eq3.5.14} and the property (c) just above, we have
\begin{equation}\label{eq3.5.17}
|\nabla h_s(z)|\leq AS_{L_s}(z)\,.
\end{equation}

Taking into account the estimate $|(c_1^1(g_{\Ga})\,,\; c_1^2(g_{\Ga}))|\leq A\om(\d)\d^2$
(see \eqref{eq3.5.7} and Definition \ref{def1}), \eqref{eq3.5.17}, \eqref{eq3.5.15} and \eqref{eq3.5.16},
we clearly can find the required $h_{\Ga}$ as an appropriate linear combination of functions $h_1$ and $h_2$.

\end{proof}

It remains to show that the function $\nabla \sum_{j\in J} f_j$ is uniformly approximated on $\Cbb$
with accuracy $A\omega(\delta)$ by the function $\nabla F$, where
$$
  F= {\sum_{n}}' \left(\sum_{j\in \Gamma^n} f^*_j +h_{\Gamma^n}\right)
  + {\sum_{n}}''\sum_{j\in\Gamma^n} f^*_j\,,
$$
where $\sum'_{n}$ and $\sum''_{n}$ are summations over all
complete and incomplete groups respectively.

For the proof of this assertion it is sufficient to check that for
each $z\in \Cbb$ we have
$$
|\nabla(F(z)-f(z))| \leq {\sum_{n}}'|\nabla (g_{\Gamma^{n}} (z) - h_{\Gamma^{n}}(z))| +
{\sum_{n}}''|\nabla g_{\Gamma^{n}}(z)|\leq A\omega(\delta)\,.
$$
After that it will be
enough let $\delta$ tend to 0.

Now our situation is absolutely analogous to that of \cite[p. 200 -- 203]{P} (2-dimensional case);
some simple details (of the following last part of the proof), dropped here, can be found there.

First, estimate the sum ${\sum_{n}}''|g_{\Gamma^{n}}(z)|$. This is very easy, because in each
$P_s(J,l)=\{j\in J\,:\, j_{3-s}=l_{3-s}\}$ we can find at most one incomplete group $\Ga$ ($s$-incomplete chain $L_s=\Ga$).
Therefore, by \eqref{eq3.5.10} and (3) of Definition \ref{def1}, we can majorize the considered sum by $A\omega(\delta)\sum_{m=1}^{+\infty}m^{-2}$, and this is it.

The estimating of ${\sum_{n}}'|\nabla (g_{\Gamma^{n}} (z) - h_{\Gamma^{n}}(z))|$ is more complicated.
For each {\it complete} group $\Gamma^{n}$ set $\chi^{n}= g_{\Gamma^{n}}
-h_{\Gamma^n}$. Then we have by \eqref{eq3.5.10} and Lemma \ref{main}
\begin{equation}\label{eq3.5.18}
|\nabla \chi^n(z)|\leq A\om(\d)S_{\Ga^n}(z)\,,\, c_0(\chi^n)=c_1^1(\chi^n)=c_1^2(\chi^n)=0\,.
\end{equation}
It suffices to prove that
$$
   {\sum_{n}}'\,|\nabla \chi^n(z)|\leq A\omega(\delta)
$$
for each $z\in\Cbb$. From now on we fix $z\in\Cbb$;
without loss of generality we can suppose that $|z|<\delta$.
All further constructions will be relative also to $z$.

Let $\Gamma^{n}=L_1^n\cup L_2^n$ (with vertex $l^n$) be a complete group. Put $a^n=a_{l^n}$,
$M^n_s=\diam (B^*_{L^n_s})/\d$, $M^n=\max\{M^n_1\,,\,M^n_2 \}$. Divide the collection  of all complete groups
in two classes.

{\it Class} (1). Here we take all complete groups $\Ga^n$ with $M^n \leq |l^n|^{1/4}$.

Clearly, the latter is possible only if $|z-a^n| \geq (|l^n|-1)\d \geq ((M^n)^4-1)\d > 2pM^n\d$.
Since $\chi^n\in A\om(\d){\cal I}(B(a^n, M^n\d)$ and \eqref{eq3.5.18} holds, we have by \eqref{eq3.4.53}:
$$
|\nabla \chi^n (z)| \leq A\om(\d)\frac{(M^n \d)^3}{(|l^n| \d)^3} \leq A\om(\d) (|l^n|^{-9/4})\,.
$$
Since in each annulus $B(0, (m+1)\d)\setminus \overline{B(0, m\d)}$ ($m>p$) we can find at most $Am$ vertices
of groups, we can see that
$$
{\sum_{n}}^{(1)}\,|\nabla \chi^n(z)|\leq A\omega(\delta)\sum_{m>p}(m^{-5/4}) \leq A_1\omega(\delta)\,,
$$
where the last sum corresponds to all complete groups of the Class (1), which now is well estimated.

{\it Class} (2). Here we place all complete groups $\Ga^n$ for which $M^n > |l^n|^{1/4}$. Fix such a group
$\Ga^n$. Then, for some $s=s^n \in \{1,2\}$ we have $M^n_s > |l^n|^{1/4}$. First we consider
the case when $s^n=1$ is {\it just one} with the last property for $\Ga^n$.
Clearly, then $|l^n| > 2p\d$, $M^n_2 \leq |l^n|^{1/4}$, so that
$$
S_{L^n_2}(z) \leq \frac{A\om (\d)\d^2}{|z-a^l|^2} \,,
$$
and then
$$
|\nabla \chi^n (z)| \leq A\om(\d)\left(\frac{\d^2}{|z-a^l|^2} + S_{L^n_1} (z) \right)\,.
$$
The same way we argue when $s^n=2$ is {\it just one} with the property $M^n_s > |l^n|^{1/4}$.

In any case we have the following lemma.
\begin{lemma}
Fix an integer $m$, and let $V_{ms}$ denote the collection of all complete groups $\Ga^n$ of the Class (2) with
$l^n_{3-s}=m$ and such that $M^n_s > |l^n|^{1/4}$. Then
$$
\sum_{n\in V_{ms}} S_{L^n_s} (z) \leq A \,,\, |m|<2p\,,
$$
and
$$
\sum_{n\in V_{ms}} S_{L^n_s} (z) \leq A m^{-5/4}\,,\, |m|\geq 2p\,.
$$
\end{lemma}
\begin{proof}
For $|m|<2p$ this follows from the estimate $||S_{P_s(J, j)}||\leq A$ for each $j\in J$. Let now $|m|\geq 2p$.
Since all $L^n_s$, $n \in V_{ms}$, are pairwise "disjoint" and $M^n_s> |l^n|^{1/4} \geq |m|^{1/4}$, we have
$$
\sum_{n\in V_{ms}} S_{L^n_s} (z)\leq A_1\sum_{\tau\in \Zbb}\frac{\d^2}{m^2\d^2 + (|m|^{1/4}\tau)^2\d^2} \leq
A_2 |m|^{-1/2} \int_0^{+\infty}\frac{dt}{(|m|^{3/4})^2+t^2} =A|m|^{-5/4}\,.
$$
\end{proof}

Summation by $m$ and $s$ now gives the desired estimate for ${\sum_{n}}^{(2)}\,|\nabla \chi^n(z)|$, corresponding to the Class (2).
This ends the proof of Theorem \ref{main-th2}.

{\section{Proof of Theorem \ref{main-th1}.}}

Observe that, by Lemma \ref{solrep} we have $(\pa_1 \Phi(z), \pa_2 \Phi(z))=k_1(1/z_1, \nu/z_2)$ if $\la_1 \neq \la_2$ or
$(\pa_1 \Phi(z), \pa_2 \Phi(z))=k_1(1/z_2, -z_1/z_2^2)$ otherwise. Define
$$\big(K_1(z), K_2(z)\big)=\Big(\frac1{z_1}, \frac1{z_2}\Big)
\quad \mbox{
if $\la_1 \neq \la_2$,}$$
and
$$\big(K_1(z), K_2(z)\big)=\Big(\frac{z_1}{z_2^2}\,, \frac1{z_2}\Big)
\quad \mbox{
if $\la_1 = \la_2$.}$$
Then, clearly, $\al_1(E)$ is comparable to
$$
\al_{12}(E) = \al_{12 \cal L}(E) =\sup_{T} \{| \langle T, 1\rangle| \,:\, {\Spt} (T) \subset E ,\, K_s * T \in C(\Rbb^2)\,,\,||K_s * T||\leq 1\,, s\in \{1, 2\}\}\,,
$$
and $\ga_1(E)$ is comparable to
$$
\ga_{12}(E) = \ga_{12 \cal L}(E) =\sup_{T} \{| \langle T, 1\rangle| \,:\, {\Spt} (T) \subset E ,\, K_s * T \in L_{\infty}(\Cbb)\,,\,||K_s * T||\leq 1\,, s\in \{1, 2\}\}\,.
$$

\subsection{Preliminaries}

We assume all measures to be positive, Borel and locally finite.
A measure $\mu$ in $\Cbb$ is said to have linear growth  (or $A_0$-linear growth) if there exists some constant $A_0>0$ such that
$$\mu(B(z,r))\leq A_0\,r\quad\mbox{ for all $z\in\Cbb$, $r>0$.}$$
The maximal Hardy-Littlewood operator with respect to $\mu$ applied to a signed measure $\nu$ is defined by
$$M_\mu \nu(z) =\sup_{r>0} \frac{|\nu|(B(z,r))}{\mu(B(z,r))}.$$
For a function $f\in L^1_{loc}(\mu)$, we write
$$M_\mu f(z) =\sup_{r>0} \frac{1}{\mu(B(z,r))}\int_{B(z,r)} |f|\,d\mu.$$
It is well known (see Chapter 2 of \cite{Mattila}, for example) that $M_\mu$ is bounded in $L^p(\mu)$ for $1<p<\infty$ and also from the space of finite
signed measures $M(\Cbb)$ into $L^{1,\infty}(\mu)$. The latter means that there exists some constant $A$ such that
$$\mu\big(\big\{z\in\Cbb:M_\mu \nu(z)>\lambda\big\}\big)\leq A\,\frac{\|\nu\|}{\lambda}\quad \mbox{for all
$\lambda>0$ and all $\nu\in M(\Cbb)$.}$$

Given a signed measure $\nu$ and a kernel $K(\cdot)$ which is $C^1$ away from the origin and satisfies
\begin{equation}\label{eq1}
|K(z)|\leq\frac{A}{|z|},\qquad |\nabla K(z)|\leq\frac A{|z|^2 }\qquad\mbox{for $z\in\Cbb\setminus\{0\}$,}
\end{equation}
we denote
$$T_K\nu(z) = \int K(z-w)\,d\nu(w)$$
whenever the integral makes sense.
 For $\varepsilon>0$ we consider the truncated version of $T_K$:
$$T_{K,\varepsilon}\nu(z) = \int_{|z-w|>\varepsilon} K(z-w)\,d\nu(w),$$
and the maximal operator
$$T_{K,*}\nu(z) = \sup_{\varepsilon>0} |T_{K,\varepsilon}\nu(z)|.$$
For a fixed positive Borel measure $\mu$ and $f\in L^1_{loc}(\mu)$, we write
$T_{K,\mu} f= T_K(f\mu)$,
$T_{K,\mu,\varepsilon}f = T_{K,\varepsilon}(f\mu)$, and $T_{K,\mu,*}f = T_{K,*}(f\mu)$.
We say that $T_{K,\mu}$ is bounded in $L^p(\mu)$ if the operators $T_{K,\mu,\varepsilon}$ are bounded uniformly on $\varepsilon>0$ in $L^p(\mu)$, and
we set
$$\|T_{K,\mu}\|_{L^p(\mu)\to L^p(\mu)} = \sup_{\varepsilon>0}\|T_{K,\mu,\varepsilon}\|_{L^p(\mu)\to L^p(\mu)}.$$

Analogously, if
$$\mu\big(\big\{z\in\Cbb:|T_{K,\varepsilon} \nu(z)|>\lambda\big\}\big)\leq A\,\frac{\|\nu\|}{\lambda}\quad \mbox{for all
$\lambda>0$, all $\nu\in M(\Cbb)$ and all $\varepsilon>0$,}$$
we say that $T_K$ is bounded from $M(\Cbb)$ into $L^{1,\infty}(\mu)$, and we denote by $\|T_{K}\|_{M(\Cbb)\to L^{1,\infty}(\mu)}$ the optimal constant $A$.

For technical reasons we need to consider also smoothly truncated operators. We fix a radial $C^\infty$ function
$\varphi$ which vanishes in $B(0,1/2)$ and equals $1$ in $\Cbb\setminus B(0,1)$, and for $\varepsilon>0$ we set
$\varphi_\varepsilon(z) = \varphi\big(\varepsilon^{-1} z\big)$. We write
$$T_{K,(\varepsilon)} \nu(z) = \int \varphi\Big(\frac {z-w}\varepsilon\Big)\,K(z-w)\, d\nu(w)$$
and
$$T_{K,(*)}\nu(z) = \sup_{\varepsilon>0} |T_{K,(\varepsilon)}\nu(z)|$$
and also $T_{K,\mu,(\varepsilon)} f = T_{K,(\varepsilon)} (f\mu)$, $T_{K,\mu,(*)} f = T_{K,(*)} (f\mu)$.
If $\mu$ has linear growth, it is immediate to check that there exists some constant $A$ depending only
on the kernel $K$ such that
\begin{equation}\label{eqdif8}
\bigl|T_{K,\mu,(\varepsilon)} f(z) - T_{K,\mu,\varepsilon} f(z)\bigr| \leq A\,M_\mu f(z)
\quad\mbox{ for all $z\in\Cbb$ and $\varepsilon>0$.}
\end{equation}
By the $L^p(\mu)$ boundedness of $M_\mu$ for $1<p<\infty$,
this implies that $T_{K,\mu}$ is bounded in $L^p(\mu)$ if and only if the operators $T_{K,\mu,(\varepsilon)}$ are bounded in $L^p(\mu)$ uniformly on $\varepsilon>0$ (under the linear growth assumption for $\mu$).

When $K$ is the Cauchy kernel, that is, $K(z) = \dfrac1{z}$, we have that $T_K$ is the Cauchy integral operator (or Cauchy transform) and we denote $T_K={\mathcal C}$
and $T_{K,\mu}={\mathcal C}_\mu$.
Notice also that the kernels $K_1$ and $K_2$ defined above
 satisfy \eqref{eq1}.

Given three pairwise points $z,w,\xi\in\Cbb$, we denote by $R(z,w,\xi)$ the radius of the circumference passing through
$z,w,\xi$, with $R(z,w,\xi)=\infty$ if these points are aligned. Their Menger curvature is $c(z,w,\xi) =
\dfrac1{R(z,w,\xi)}.$ If two among the points coincide or the points are aligned, we write $c(z,w,\xi)=0$.
The curvature of the measure $\mu$ is defined by
$$c^2(\mu) = \iiint c(z,w,\xi)^2\,d\mu(z)\,d\mu(w)\,d\mu(\xi).$$
This notion was introduced by Mark Melnikov in \cite{Melnikov} while studying a discrete version of analytic capacity.

For a given compact set $E\subset\Cbb$, let $\Sigma(E)$ be the set of Borel measures supported on $E$ such that
$$\mu(B(z,r))\leq r\qquad\mbox{for all $z\in\Cbb$, $r>0$.}$$
Also, let $\Sigma_0(E)$ be the set of Borel measures $\mu\in\Sigma(E)$ such that
$$\lim_{r\to0}\frac{\mu(B(z,r)}r =0\qquad \mbox{for all $x\in\spt\mu$}.$$
In \cite{Tolsa-sem} it was shown
\begin{align}\label{eqgam}
\gamma(E) &\asymp \sup\bigl\{\mu(E):\mu\in\Sigma(E), \,\|{\mathcal C}_\mu\|_{L^2(\mu)\to L^2(\mu)}\leq 1\bigr\}\\
& \asymp \sup\bigl\{\mu(E):\mu\in\Sigma(E), \,c^2(\mu)\leq\mu(E)\bigr\},\nonumber
\end{align}
and in \cite{tol04},
\begin{align}\label{eqalph}
\alpha(E) &\asymp \sup\bigl\{\mu(E):\mu\in\Sigma_0(E), \,\|{\mathcal C}_\mu\|_{L^2(\mu)\to L^2(\mu)}\leq 1\bigr\}\\
& \asymp \sup\bigl\{\mu(E):\mu\in\Sigma_0(E), \,c^2(\mu)\leq\mu(E)\bigr\}.\nonumber
\end{align}

Another result that will be needed for the proofs of $\gamma_{12}\gtrsim\gamma$ and $\alpha_{12}\gtrsim
\alpha$ is the following:

\begin{theor}[\cite{Tolsa-pubmat}] \label{teopubmat}
Let $\mu$ be a locally finite Borel  measure without point masses in $\Cbb$.
If the Cauchy transform ${\mathcal C}_\mu$ is bounded in $L^2(\mu)$, then any
singular integral operator $T_{K,\mu}$ associated with an odd kernel $K\in C^\infty(\Cbb\setminus\{0\}$ satisfying
$$
|z|^{1+j}\,|\nabla^j K(z)|\in L^\infty(\Cbb)\quad\mbox{ for all $z\neq0$ and
$j=0,1,2,\ldots$}
$$
is also bounded on $L^2(\mu)$. Further, the norm of the operator $T_{K,\mu}$ in $L^2(\mu)$ is bounded by some constant
depending only on the one of ${\mathcal C}_\mu$ as an operator in $L^2(\mu)$ and on the numbers $\sup_{z\neq0}|z|^{1+j}\,|\nabla^j K(z)|$, $j=0,1,2\ldots$.
\end{theor}

\vspace{2mm}
\subsection{Proof of $\gamma\gtrsim \gamma_{12}$ and $\alpha\gtrsim \alpha_{12}$}
Let $E\subset\Cbb$ be compact.
By the definition of $\gamma_{12}$, there exists a distribution $T$ supported on $E$ such that $\|K_s*T\|_{L^\infty(\Cbb)}\leq1$ for $s=1,2$ and $\gamma_{12}(E)\leq 2|\langle T,1\rangle|$, with $K_1,K_2$ as above.
In particular, we have $\|K_2*T\|_{L^\infty(\Cbb)}\leq 1$.

Consider the non-degenerate linear map in $\Cbb$ defined by $\La(z) = z_2$.
Let $\La_{\sharp} T$ be the push-forward distribution defined by
$$\langle \La_{\sharp} T,\varphi\rangle = \langle T,\varphi\circ \La\rangle\qquad \mbox{for all $\varphi\in C^\infty(\Cbb)$.}$$
Notice that $\La_{\sharp} T$ is supported on $\La(E)$, and for the kernel
$$\widetilde K_2(z) = K_2(\La^{-1}(z))\qquad \mbox{for all $z\in\Cbb\setminus\{0\}$}$$
it is easy to check that
$$((\La_{\sharp} T) * \widetilde K_2)(z) = (T* K_2)(\La^{-1}(z)).$$
Observe that $\widetilde K_2(z)$ coincides with the Cauchy kernel $\dfrac1z$ and thus, by the definition of analytic capacity,
\begin{equation}\label{eqgam122}
\gamma(\La(E))\geq |\langle (\La_{\sharp} T),1\rangle| = |\langle T,1\rangle| \geq 2^{-1} \gamma_{12}(E),
\end{equation}
Now we could use the fact that, by \cite{Tolsa-bilip} for any bilipschitz mapping $f:\Cbb\to\Cbb$ one has $\gamma(f(F))\asymp\gamma(F)$ for any compact set $F$, with the comparability constant depending just on the bilipschitz constant, and so
$$\gamma(\La(E))\asymp\gamma(E),$$
concluding the proof of $\gamma\gtrsim\gamma_{12}$.

An alternative argument which exploits the fact that $\La$ is a linear map is the following: by \eqref{eqgam} we know that
there exists some measure $\mu\in\Sigma(\La(E))$ such that $\mu(\La(E))\asymp\gamma(\La(E))$ and  $c^2(\mu)\leq \mu(\La(E))$.
From the fact that
$$R(z,w,\xi)\asymp R(\La(z), \La(w), \La(\xi))\qquad\mbox{for all $z,w,\xi\in\Cbb$},$$
with the comparability constant depending just on $\La$, we infer that the push-forward measure $\sigma=(\La^{-1})_\sharp\mu$
satisfies $c^2(\sigma)\asymp c^2(\mu)$. Further, it is also easy to check that
$$\sigma(B(z,r))\leq A_1\,r \qquad\mbox{for all $z\in\Cbb$}.$$
Thus, for a suitable constant $c_0>0$ depending just on $\La$, $c_0\sigma\in\Sigma(E)$ and $c^2(c_0\sigma)\leq c_0\sigma(E)$.
Hence applying again \eqref{eqgam}, we derive
$$\gamma(E) \gtrsim \sigma(E) = \mu(E)\gtrsim \gamma(\La(E)),$$
which together with \eqref{eqgam122} yields $\gamma(E)\gtrsim\gamma_{12}(E)$.

The proof of the fact that $\alpha\gtrsim \alpha_{12}$ is almost the same. The only required change is that above we have to require
the function $K_s*T$ to be continuous, which in turn implies that $(\La_{\sharp} T) * \widetilde K_2$ is continuous, and thus
\eqref{eqgam122} holds with $\gamma$ and $\gamma_{12}$ replaced by $\alpha$ and $\alpha_{12}$, respectively. Then one concludes
in the same way applying either the fact that $\alpha(\La(E))$ is comparable to $\alpha(E)$ because $\La$ is bilipschitz (by applying \cite{Tolsa-bilip} to $\alpha$), or by following the last alternative argument, using \eqref{eqalph} instead of \eqref{eqgam}.
\fiproof

\vspace{2mm}

\subsection{Proof of $\gamma\lesssim \gamma_{12}$ }

We need some auxiliary lemmas. The first one is based on some work which goes back to Davie and Oksendal, and its proof can be found with minor modifications in \cite{Mattila-Paramonov} (see Lemma 4.2 there and the definitions of the standard notations
${\mathcal{M}}(X)$, $C_0(X)$, $T_j^*$ just before it).

\begin{lemma}\label{christ}
Let $\mu$ be a Radon measure on a locally compact Hausdorff space X and let
$T_j: {\mathcal{M}}(X)\rightarrow C_0(X),\
j=1,...,d$ be linear bounded operators. Suppose that each transpose $T^*_j:{\mathcal{M}}(X)\rightarrow C_0(X)$ is bounded from ${\mathcal{M}}(X)$ to $L^{1,\infty}(\mu)$, that is to say that there exists a constant A such that
\begin{equation} \mu \{x: |T_j^* \nu (x)|>\lambda \} \leq C \frac{\| \nu
\|}\lambda \end{equation}
for $j=1,...,d$, $\lambda>0$ and $\nu \in
{\mathcal{M}}(X)$. Then, for each $\tau >0$ and any Borel set
$E\subset X$ with $0<\mu (E)<\infty$, there exists $h: X \to [0, 1]$ in $L^{\infty}(\mu)$
satisfying $h(x)=0$ for $x\in X\backslash E$,
$$ \int_E h\,d\mu \geq \frac1{1+\tau}\,\mu(E)$$ and
$$ \|T_j(h\mu)\|_\infty \leq A(C,\tau,d),\qquad \mbox{ for $j=1,...,d$.}$$\\
\end{lemma}
From this lemma we get the following.
\begin{lemma}\label{lemoks}
Let $\mu$ be a measure in $\Cbb$ with compact support which has $A_0$-linear growth.
For $j=1,\ldots,d$, let $K^j$ be a kernel satisfying the conditions in \eqref{eq1}, and denote
$T^j=T_{K^j,\mu}$, $T^j_{(\varepsilon)}=T_{K^j,\mu,(\varepsilon)}$.
Suppose that $T^j$ is bounded in $L^2(\mu)$, with norm at most $C$. Then for each $\tau >0$
there exists some function $h:\spt\mu\to[0,1]$ such that
$$
\int h\,d\mu \geq \frac1{1+\tau}\,\|\mu\|,
$$
$$ \|T^jh\|_{L^\infty(\Cbb)} \leq A(A_0,C,\tau,d)\qquad \mbox{ for $i=1,...,d$},$$
and
$$ \|T^j_{(\varepsilon)}h\|_{L^\infty(\Cbb)} \leq A(A_0,C,\tau,d)\qquad \mbox{ for $i=1,...,d$ and all $\varepsilon>0$.}$$
\end{lemma}

\begin{proof}
The arguments are very standard and we just sketch them. Since $T^j\equiv T_{K^j,\mu}$
is bounded in $L^2(\mu)$,
then $T_{K^j}$ (and its transpose) is bounded from ${\mathcal{M}}(\Cbb)$ into $L^{1,\infty}(\mu)$ (see for example \cite[Chapter 2]{Tolsa-book}).
Then we apply Lemma \eqref{christ} to each smoothly truncated operator $T^j_{(\varepsilon)}$ and we deduce the existence of a function $h_{\varepsilon}:\spt\mu\to[0,1]$
such that $$ \int h_\varepsilon\,d\mu \geq \frac1{1+\tau}\,\|\mu\|$$ and
$$ \|T_j(h_{\varepsilon} \mu)\|_{L^\infty(\Cbb)} \leq A(A_0,C,\tau)\qquad \mbox{ for $j=1,...,d$.}$$
By a compactness argument in weak $L^\infty(\Cbb)$ we deduce the existence of a single
function $h$ fulfilling the properties of the lemma.
\end{proof}

We are ready to prove that $\gamma\lesssim \gamma_{12}$ now.
Let $E\subset\Cbb$ be compact. By \eqref{eqgam} there exists a measure $\mu\in\Sigma(E)$ such that $\|{\mathcal C}_\mu\|_{L^2(\mu)\to L^2(\mu)}\leq 1$ and
and $\gamma(E)\asymp\mu(E)$. By Theorem \ref{teopubmat}, $T_{K_1,\mu}$ and $T_{K_2,\mu}$ are bounded in
$L^2(\mu)$. By (non-homogeneous) Calder\'on-Zygmund theory, then the operators $T_{K_s}$ are
bounded from  $M(\Cbb)$ to $L^{1,\infty}(\mu)$ for $s=1,2$ (see \cite[Chapter 2]{Tolsa-book}, for example).
Then by Lemma \ref{lemoks} there exists some function $h:E\to[0,1]$
such that $$ \int h\,d\mu \geq \frac1{2}\,\|\mu\|$$ and
$$ \|T_{K_s,\mu}h\|_{L^\infty(\Cbb)} \leq A\qquad \mbox{ for $s=1,2$.}$$
As a consequence, from the definition of $\gamma_{12}$ we deduce that
$$\gamma_{12}(E)\gtrsim \int h\,d\mu\gtrsim \mu(E)\gtrsim\gamma(E),$$
which completes the proof of $\gamma_{12}\gtrsim \gamma$.
\fiproof

\subsection{Proof of $\alpha\lesssim \alpha_{12}$ } \label{sek}

Let $E\subset\Cbb$ be compact. By \eqref{eqalph} there exists a measure $\mu\in\Sigma_0(E)$ such that $\|{\mathcal C}_\mu\|_{L^2(\mu)\to L^2(\mu)}\leq 1$ and
and $\alpha(E)\asymp\mu(E)$.
Again by Theorem \ref{teopubmat}, we know that $T_{K_1,\mu}$ and $T_{K_2,\mu}$ are bounded in
$L^2(\mu)$. Our next objective is to find some function
$h:E\to[0,1]$, supported on $E$, such that $\int h\,d\mu\geq c\,\mu(E)$ (with $c = c({\mathcal L})>0$),
\begin{equation}\label{eqda1}
\|T_{K_s,\mu,\varepsilon}h\|_{L^\infty(\Cbb)}\leq 1 \qquad\mbox{for $s=1,2$ and all $\varepsilon>0$,}
\end{equation}
\begin{equation}\label{eqda2}
\|T_{K_s,\mu}h\|_{L^\infty(\Cbb)}\leq 1 \qquad\mbox{for $s=1,2$,}
\end{equation}
and such that moreover both $T_{K_1,\mu}h$ and $T_{K_2,\mu} h$ can be extended continuously to the whole $\Cbb$.
Note that once we prove the existence of $h$ we are done, because from the definition of $\alpha_{12}$ we deduce that
$$\alpha_{12}(E)\gtrsim \int  h\,d\mu\gtrsim \mu(E)\gtrsim\alpha(E),$$
as wished.

We will need a couple of additional auxiliary lemmas. The following result is proven (in more generality) in \cite[Lemma 3]{Ruiz-Tolsa}. The same result had been proved previously
 for the Cauchy transform in \cite{tol04}.

\begin{lemma}\label{norma}
Let $\mu$ be a measure in $\Cbb$ with compact support and linear growth and suppose that
$\lim_{r\to 0}\frac{\mu(B(x,r))}{r}=0$ for all $x \in \spt\mu$. Let $K$ be an odd kernel (i.e.,\ $K(-z)=-K(z)$ for all $z\neq0$) satisfying the conditions in \eqref{eq1}.
Suppose that $T_{K,\mu}$ is bounded in $L^2(\mu)$.
Then, given $\delta>0$ we can find $F\subset \spt\mu$ with $\mu(\mathbb C\setminus F)<\delta$ such that
\begin{itemize}
\item[(a)] $\lim_{r \to 0}{\dfrac{\mu(B(x,r)\cap F)}{r}}=0$ uniformly on $x \in \Cbb$,
\item[(b)] $\lim_{r \to 0} \|T_{K,\mu} \|_{L^2(\mu|B(x,r)\cap F) \to L^2(\mu|B(x,r)\cap
    F)}=0$, uniformly on $x \in \Cbb.$
    \end{itemize}
\end{lemma}

The next lemma is proven in \cite[Lemma 8]{Ruiz-Tolsa}, in  a more general context too.

\begin{lemma}\label{lemsuc}
Let $\mu$ be a measure in $\Cbb$ with compact support and linear growth.
Let $F=\spt\mu $ and, for $j=1,\ldots,d$, let $K^j$ be a kernel satisfying the conditions in \eqref{eq1}, and denote
$T^j=T_{K^j,\mu}$, $T^j_{(\varepsilon)}=T_{K^j,\mu,(\varepsilon)}$ and $T^j_{(*)} = T_{K^j,\mu,(*)}$.
 Suppose that
\begin{itemize}
\item [(a)]$\theta_\mu(x)=\lim_{r\rightarrow 0}\dfrac{\mu(B(x,r))}{r}=0$
uniformly on $x \in F$,
\item [(b)]$\lim_{r \rightarrow 0} \|T^j
\|_{L^2(\mu|B(x,r))\to L^2(\mu|B(x,r))}=0$ uniformly on $x \in
F$ for $j=1,\ldots,d$.
\end{itemize}
Let $f$ be a bounded function supported on F such that
$\|T^j_{(*)}f\|_{L^\infty(\Cbb)}<\infty$ for all
$j,\, \varepsilon > 0$. Then,
given $0<\tau \leq 1$, there exists $\delta>0$ and a function $g$
supported on F satisfying, for each $j=1,...,d$,
\begin{itemize}
\item [(i)]
 $\int g\,d\mu=\int f\,d\mu$ and $0\leq g \leq
\|f\|_{L^\infty (\mu)}+\tau,$

\item[(ii)] \, $\|T_{(*)}^j g\|_{L^\infty(\Cbb)}
\leq \|T_{(*)}^j f\|_{L^\infty(\Cbb)}+\tau$,

\item[(iii)]
\begin{equation*}\label{succ1}
|T^j_{(\varepsilon)} g(x)-T^j_{(\varepsilon)} g(y)|\leq \tau, \quad \mbox{if $|x-y|\leq \delta$ and $\varepsilon>0$},
\end{equation*}

\item[(iv)]
and
\begin{equation*}\label{succ2}
|T^j_{(\varepsilon)} g(x)-T^j_{(\varepsilon)} g(y)|\leq \sup_{\varepsilon'>0}|T^j_{(\varepsilon')}
f(x)-T^j_{(\varepsilon')} f(y)|+ \tau, \quad \forall \, x,y \in \Cbb,
\varepsilon>0.
\end{equation*}
\end{itemize}
\end{lemma}

Let us remark that, in fact, in \cite[Lemma 8]{Ruiz-Tolsa} the lemma above is stated for the ``sharply'' truncated operators $T^j_{\varepsilon}$, instead of smoothly truncated ones $T^j_{(\varepsilon)}$. However, the same proof also works for the operators $T^j_{\varepsilon}$. An analog of the lemma above for the Cauchy transform
(with $d=1$) with smooth truncations is proven in \cite[Lemma 3.3]{tol04}.
\vspace{2mm}
\vspace{2mm}

\noindent{\bf Construction of $h$.}
We follow quite closely the arguments in \cite[Lemma 3.4]{tol04}.
Let $\mu$ be as above, so that $\lim_{r\to 0}\frac{\mu(B(z,r))}r = 0$ for all $z\in\spt\mu$  and $T_1=T_{K_1,\mu}$ and $T_2=T_{K_2,\mu}$ are bounded in
$L^2(\mu)$.
By Lemma \ref{norma} there are subsets $F_j \subset E$ such that
$\lim_{r \rightarrow 0}{\frac{\mu(B(x,r)\cap F_j)}{r}}=0$
uniformly on $x \in \Cbb$ and
\begin{equation}\label{normz}
\lim_{r \rightarrow 0} \|T_{K_j,\mu,\varepsilon} \|_{L^2(\mu|B(x,r)\cap
F_j) \to L^2(\mu|B(x,r)\cap F_j)}=0,
\end{equation}
uniformly on $x \in \Cbb$ for $j=1,2$, such that the set
$F:=\cap F_j$ satisfies $\mu(F)\geq \mu(E)/2$.

By Lemma \ref{lemoks} there exists function $h_1$ supported on $F$, with
$0\leq h_1\leq 1$, $\|T_{j,(*)} h_1\|_{L^\infty(\Cbb)}\leq 1$ for $j=1,2$, and $\int
h_1\,d\mu\geq A^{-1}\mu(F)$. We set $\delta_1=1$.

For $n\geq1$, we set $\tau_n=2^{-n}$, and  given a positive
bounded function $h_n$ supported on $F$ and $\delta_n>0$, by means of
Lemma \ref{lemsuc} we construct a function $h_{n+1}$ also supported
on $F$, so that $\int h_{n+1}\,d\mu = \int
h_n\,d\mu$, $0\leq h_{n+1}\leq \|h_n\|_{L^\infty(\mu)}+\tau_n$, such that for $j=1,2$,
$\|T_{j,(*)} h_{n+1}\|_{L^\infty(\Cbb)}\leq \|T_{j,(*)} h_n\|_{L^\infty(\Cbb)} +
\tau_n$, and moreover
\begin{equation}
\label{wwcond1'} |T_{j,(\varepsilon)} h_{n+1}(x) - T_{j,(\varepsilon)} h_{n+1}(y)| \leq \tau_n
\quad \mbox{for $|x-y|\leq\delta_{n+1}$ and all $\varepsilon>0$}
\end{equation}
(where $\delta_{n+1}\leq \delta_n$ is some constant small enough),
and
\begin{equation} \label{wwcond2'} |T_{j,(\varepsilon)} h_{n+1}(x) - T_{j,(\varepsilon)} h_{n+1}(y)| \leq
\sup_{\varepsilon'>0} |T_{j,(\varepsilon')} h_n(x) - T_{j,(\varepsilon')} h_n(y)| + \tau_n
\end{equation}
for all $x,y\in\Cbb$, $\varepsilon>0$.

Let $h$ be a weak $*$ limit in $L^\infty(\mu)$ of a subsequence
$\{h_{n_k}\}_k$. Clearly, $h$ is a positive bounded function such that
\begin{equation}\label{eqfi10}
\int h\,d\mu =\int h_1\,d\mu \gtrsim\mu(F)\gtrsim \alpha(E).
\end{equation}
Also, for each $\varepsilon>0$ and $x\in \Cbb$,
$$
T_{j,(\varepsilon)} h_{n_k}(x) \to T_{j,(\varepsilon)}
h(x)\quad \mbox{as}\quad k\to\infty.
$$
Since
$$\|T_{j,(\varepsilon)} h_n\|_{L^\infty(\Cbb)} \leq \|T_{j,(*)} h_1\|_{L^\infty(\Cbb)} +
\sum_{i=1}^{n-1} 2^{-i}\leq 2$$ for all $n$, we deduce $\|T_{j,(\varepsilon)}
h\|_{L^\infty(\Cbb)} \leq 2$,  which implies that
\begin{equation}\label{eqfi11}
\|T_{j,\e}
h\|_{L^\infty(\Cbb)} \lesssim 1\quad\mbox{ for all $\e>0$,}
\end{equation}
by \eqref{eqdif8}.

On the other hand, by \eqref{wwcond1'}, if $|x-y|\leq \delta_n$, then
$$|T_{j,(\varepsilon)} h_n(x) - T_{j,(\varepsilon)} h_n(y)|\leq2^{-n}$$
for all $\varepsilon>0$. From \eqref{wwcond2'}, for $k\geq n$ we get
$$|T_{j,(\varepsilon)} h_k(x) - T_{j,(\varepsilon)} h_k(y)| \leq \sup_{\varepsilon'>0}|T_{j,(\varepsilon')} h_n(x) - T_{j,(\varepsilon')}
h_n(y)| + \sum_{i=n}^{k-1} 2^{-i} \leq 2^{-n+2},$$ assuming
$|x-y|\leq \delta_n$. Thus,
\begin{equation} \label{eqcont}
|T_{j,(\varepsilon)} h(x) - T_{j,(\varepsilon)} h(y)| \leq 2^{-n+2} \quad \mbox{if $|x-y|\leq
\delta_n$.}
\end{equation}

Consider now the family of functions $\{T_{j,(\varepsilon)} h\}_{\varepsilon>0}$ on
$\bar{B}(0, R)$, where $R$ is big enough so that $E\subset
B(0,R-1)$. This is a family of functions which is uniformly
bounded and equicontinuous on $\bar{B}(0,R)$, by \eqref{eqcont}. By
the Ascoli-Arzel\`a theorem, there exists a sequence
$\{\varepsilon_n\}_n$, with $\varepsilon_n\to 0$, such that $T_{j,(\varepsilon_n)} h$
converges uniformly on $\bar{B}(0,R)$ to some continuous function
$g_j$. It is easily seen that $g_j$ coincides with $T_j h$ ${\mathcal L}^2$-a.e.
in $\bar{B}(0,R)$ and
$\|g_j\|_{L^\infty(\bar{B}(0,R))}\leq2$. By continuity, it is clear then that $g_j$ and $T_j\mu$ coincide 
$\bar{B}(0,R)\setminus E$.
Since $T_{j} h$ is also
continuous in $\Cbb\setminus \bar{B}(0,R-\frac12)$, we deduce that the function which equals $g_j$ on $E$ and $T_j h$ in the complement of $E$ is continuous in
 the whole complex plane, as wished. Together with  \eqref{eqfi10} and \eqref{eqfi11}, this
 shows that $
 c_0\,h$ satisfies the required properties stated at the beginning of this section, for some constant $c_0>0$ depending at most on $K_1$ and $K_2$.
\fiproof

{\section{The $C^1$-approximation criteria for classes of functions and the $C^1$-approximation
by ${\cal L}$-polynomials}}

It is worth mentioning (see, for instance, \cite[Theorem 1.12]{mpf2012} and it's proof) that
the zero sets for capacity $\al_1$ (or $\al$) are precisely the sets of $C^1$-{\it removable singularities}
for solutions of the equation ${\cal L}u = 0$.

Standard arguments (see, for instance, \cite[Proof of Theorem 6.1]{P}) allow to deduce from Theorem \ref{main-th} or Theorem \ref{main-th2}
the following $C^1$-approximation criterion for "classes of functions".

\begin{theor}\label{ForClasses} For a compact set $X$ in $\Cbb$ the following conditions are equivalent:
 \begin{itemize}
 \item[(a)] $\; {\cal A}^1_{\cal L}(X) = C^1_{\cal L}(X)\,$;
 \item[(b)] $\; \al_1(D\setminus X^0)=\al_1(D\setminus X)$ for any bounded open set $D\,$;
 \item[(c)] there exist $A>0$ and $k\geq 1$ such that
$$
\al_1(B(a,\delta)\setminus X^0)\leq A\al_1(B(a,k\delta)\setminus X)
$$
for each disc $B(a,\delta)\,$.
\end{itemize}
\end{theor}
From the last theorem, Theorem \ref{main-th1} and Vitushkin's criteria for uniform rational approximations
(see \cite[Ch. V, \S 2-3]{Vi} or \cite[\S 1]{P95}, including the definitions of $R(X)$ and $A(X)$) the next corollary follows directly.

\begin{corollary} For a compact set $X$ in $\Cbb$ the following conditions are equivalent:
 \begin{itemize}
 \item[(a)] $\; {\cal A}^1_{\cal L}(X) = C^1_{\cal L}(X)\,$;
 \item[(b)] $R(X)=A(X)\,$.
\end{itemize}
\end{corollary}

Applying \cite[Theorem 1.3]{{tol04}} and the previous result, we obtain also the following corollary.

\begin{corollary}
Let $X$ be a compact set in $\Cbb$ with inner boundary $\pa_i X$. If $\al(\pa_i X)=0$
then $\; {\cal A}^1_{\cal L}(X) = C^1_{\cal L}(X)\,$.
\end{corollary}

For a compact set $X$ in $\Cbb$ and a function $f$ of class $C^1$ in some neighbourhood of $X$, define
the $C^1$-Whitney norm of $f$ on $X$ \cite{Wh}:
$$
||f||_{1X} = \inf\{||F||_1\,:\, F \in BC^1(\Cbb)\,,\, F|_X=f|_X\,,\, \nabla F|_X = \nabla f|_X\}\,.
$$
Denote by ${\cal P}_{\cal L}$ the space of all polynomials $p$ of real variables, such that ${\cal L}p \equiv 0$
(see \cite[Proposition 2.1]{PF99}).
The following analog of the well known Mergelyan theorem \cite{Mer}  is a direct corollary of Theorems \ref{main-th1}
and \ref{ForClasses}, the fact that $\al_1(D)\asymp \diam (D)$ for any domain $D$ (because the same holds for the capacity $\alpha$), and Runge-type theorems (see \cite[Theorem 3 and Proposition 2]{BP}).
\begin{theor} For a compact set $X$ in $\Cbb$ the following conditions are equivalent:
 \begin{itemize}
 \item[(a)] for each $f \in C^1_{\cal L}(X)\,$ and $\e > 0$ there is $p \in {\cal P}_{\cal L}$ with $||f-p||_{1X} < \e\,$;
 \item[(b)] $\Cbb \setminus X$ is connected.
\end{itemize}
\end{theor}

This result strengthens \cite[Theorem 1.1 (4)]{PF99}, since the norm $||\cdot||_{1X}$ is stronger than the norm
$$
||f||_{1Xw} = \max\{||f||_X\,,\, ||\nabla f||_X \}\,,
$$
considered in \cite{PF99}.

As a plan for our subsequent work in this themes we formulate the following conjecture.
\begin{conj}
Theorems \ref{main-th1} and \ref{main-th} have their direct analogs for all dimensions
(in $\Rbb^N$ for all $N \in \{3, 4, \dots \}$). In Theorem \ref{main-th1} we just have
to take, instead of the capacities $\al$ and $\ga$, the $C^1$- and $\Lip^1$-harmonic capacities
respectively (see \cite{Ruiz-Tolsa}).
\end{conj}


\medskip

\noindent {\bf P.V. Paramonov \\

\medskip
\noindent
Mechanics and Mathematics Faculty of\\
Moscow State University.\\
119991 Moscow, Russian Federation.\\

\medskip
\noindent
Mathematics and mechanics Faculty of\\
St-Petersburg State University. \\
Saint Petersburg, Russian Federation.\\

\medskip
\noindent
\tt e-mail: petr.paramonov@list.ru\\

\bigskip\bigskip

\noindent {\bf X. Tolsa \\

\medskip
\noindent
ICREA, \\
Passeig Llu\'{\i}s Companys 23 08010 Barcelona, Catalonia,\\

\medskip
\noindent
Departament de Matem\`atiques and BGSMath,\\
Universitat Aut\`onoma de Barcelona
\\
08193 Bellaterra (Barcelona), Catalonia

\medskip
\noindent
\tt e-mail: xtolsa@mat.uab.cat}

\end{document}